\documentclass{amsart}
\usepackage{amsmath,amsthm}
\usepackage{amsfonts,amssymb}
\usepackage{accents}
\usepackage{enumerate}
\usepackage{accents,color}
\usepackage{graphicx}
\usepackage{todonotes}

\newcommand{\lvt}{\left|\kern-1.35pt\left|\kern-1.3pt\left|}
\newcommand{\rvt}{\right|\kern-1.3pt\right|\kern-1.35pt\right|}
\newtheorem{thm}{Theorem}[section]
\newtheorem{cor}[thm]{Corollary}

\newtheorem{prop}[thm]{Proposition}

\theoremstyle{remark}

 \def\la{{\langle}}
 \def\ra{{\rangle}}

 \def\a{{\alpha}}
 \def\b{{\beta}}
 \def\g{{\gamma}}
 
 \def\t{{\theta}}
 \def\l{{\lambda}}

 \def\s{\sigma}
 \def\la{{\langle}}
 \def\ra{{\rangle}}

 \def\CH{{\mathcal H}}

 \def\CL{{\mathcal L}}
 
 \def\CP{{\mathcal P}}

 \def\NN{{\mathbb N}}

 \def\RR{{\mathbb R}}

\def\lla{\langle{\kern-2.5pt}\langle}      
\def\rra{\rangle{\kern-2.5pt}\rangle}

\newcommand{\wh}{\widehat}

\def\f{\frac}

\graphicspath{{./}}
\begin{document}
 
\title{Orthogonal structure on a quadratic curve}

\author{Sheehan Olver}
\address{Department of Mathematics\\
Imperial College\\
 London \\
 United Kingdom  }\email{s.olver@imperial.ac.uk}

\author{Yuan Xu}
\address{Department of Mathematics\\ University of Oregon\\
    Eugene, Oregon 97403-1222.}\email{yuan@uoregon.edu}

\thanks{The first author was supported in part by a Leverhulme Trust Research Project Grant and the second author was supported in part by NSF Grant DMS-1510296. Both authors would like to thank the Isaac Newton Institute for Mathematical Sciences for support and hospitality during the programmes ``Approximation, sampling and compression in data science'' and ``Complex analysis: techniques, applications and computations''  when work on this paper was  undertaken. This work was supported by EPSRC grant number EP/R014604/1.}

\date{\today}

\begin{abstract} 
Orthogonal polynomials on quadratic curves in the plane are studied. These include orthogonal polynomials
on ellipses, parabolas, hyperbolas, and two  lines. For  an integral with respect 
to an appropriate weight function defined on any quadratic curve, an explicit basis of orthogonal polynomials is 
constructed in terms of two families of orthogonal polynomials in one variable. 
Convergence of the Fourier orthogonal expansions is also studied in each case.  We discuss applications 
to the Fourier extension problem, interpolation of functions with singularities or near singularities, and
the solution of Schr\"odinger's equation with non-differentiable or nearly-non-differentiable potentials. 
\end{abstract}

\maketitle

\section{Introduction}
\setcounter{equation}{0}

The purpose of this work is to study orthogonal polynomials of two variables  defined 
on a quadratic curve in the plane. Every quadratic curve in the plane satisfies an equation of the form
$$
     \gamma(x,y):=  a_{1,1} x^2 + 2 a_{1,2} xy + a_{2,2} y^2 + 2 a_{1,3} x + 2 a_{2,3} y + a_{3,3}  =0
$$
for some constants $a_{i,k}$.
Let  $\Omega = \Omega_\gamma:= \{(x,y): \gamma(x,y) = 0\} \subset \RR^2$ be  the trace of the curve in 
the plane. Let $\varpi(x,y)$ be a nonnegative weight function defined on $\Omega$ and satisfying
$\int_\Omega \varpi(x,y) d\s(x,y) > 0$, where $d\s$ is the Lebesgue measure on the boundary $\Omega$. We consider the bilinear form defined by 
\begin{equation}\label{eq:ipd}
  \la f, g \ra = \int_{\Omega} f(x,y) g(x,y) \varpi(x,y) d \s(x,y).
\end{equation}
 Since $\la f, g \ra =0$ 
if either $f$ or $g$ belongs to the polynomial ideal $\langle \gamma \rangle$, it is an inner product only 
on the space $\RR[x,y]/ \langle \gamma \rangle$; that is, polynomials modulo the ideal $\la \gamma \ra$. 

When $\gamma(x,y) = 1-x^2-y^2$, so that $\Omega$ is the unit circle, the orthogonal structure is well-understood. 
In particular, when $\varpi(x,y) =1$, the orthogonal polynomials are spherical harmonics, 
which are homogeneous polynomials of two variables, their restriction on the circle 
are $\cos (n \t)$ and $\sin (n \t)$ where $x = \cos \t$ and $y = \sin \t$, and the corresponding Fourier orthogonal expansions are the usual Fourier 
series. This case has served as the model for our recent study \cite{OX} of orthogonal polynomials on a 
wedge, defined as two line segments that share a common endpoint. In this paper we shall construct orthogonal 
polynomials for other quadratic curves. 
 
Under an affine change of variables, quadratic curves in the plane are classified into five classes; three
non-degenerate cases: ellipses, parabolas, and hyperbolas; and two degenerate cases: two intersecting real 
lines, and two parallel real lines. Since an affine change of variables does not change the degree of 
a polynomial in two variables, we only need to study orthogonal polynomials on a standard quadratic curve
for each class of quadratic curves. Our choice of  standard quadratic curves are the following: 
\begin{enumerate} [\qquad (i)]
\item ellipse: $x^2+y^2 =1$; 
\item parabola: $y = x^2$;
\item hyperbola: $x^2 - y^2 =1$;
\item two intersection lines: $x=0$ and $y=0$; and,
\item parallel lines: $y=-1$ and $y=1$.
\end{enumerate}

%\sotodoinline{I'm not sure this next paragraph contributes that much to the introduction. Should we remove it?} 
%A wedge can be considered as a special case of two intersection lines. Indeed, if we choose our weight
%function $w$ to have its support set on $\Gamma = \{(x,y), x \ge 0, \, y \ge 0\}$, then the inner product 
%in the case (iv) becomes one on the wedge $\Gamma$. Clearly, up to affine transforms, we only need to 
%consider orthogonal polynomials on one wedge. In \cite{OX}, we choose the wedge as 
%$\{(x,1): 0 \le x \le 1\} \cup \{(y,1): 0 \le  y\le 1\}$, since we had in mind to use the results to construct
%orthogonal polynomials on the boundary of a square, but the results there can be easily transformed to
%the wedge $\Gamma$.  

As in the known  cases of orthogonal polynomials on the unit circle and on a wedge, the space of 
orthogonal polynomials of degree $n$ has dimension $2$ when $n \ge 1$ for each quadratic curve. 
Assuming that $\varpi$ satisfies certain symmetry properties, the two-dimensional
space of orthogonal polynomials on each quadratic curve can be expressed in terms of
two different families of orthogonal polynomials in one variable, whose structure is well-understood.
Furthermore, we shall relate the Fourier orthogonal expansions on the quadratic curve to those on the real
line so that the convergence of the former can be deduced from the latter. Our approach will follow closely 
the development in \cite{OX} and relies heavily on symmetry of the domain and the weight function.  

We consider three applications of the results. The first is usage in the Fourier extension problem, essentially
rephrasing the known results of \cite{MH,H} in terms of the orthogonal polynomials on quadratic curves.
The second  application is interpolation of functions on the real line with singularities of the form $|x|$, 
$\sqrt{x^2+ \epsilon^2}$, and $1/x$. 
For example, we consider functions  on the interval $[-1,1]$ of the form
$$
f(x) = f(\sqrt{x^2 + \epsilon^2}, x)
$$
where $f(x,y)$, defined on the hyperbola $x^2 = y^2 + \epsilon^2$, is entire in $x$ and $y$. 
We will see that we can calculate the two-variable polynomial interpolant of $f(x,y)$ using   orthogonal polynomials
via a suitably constructed quadrature rule. The resulting approximation
converges super-exponentially fast,  at a rate that is uniformly bounded as $\epsilon \rightarrow 0$.
This is in  contrast to using only univariate orthogonal polynomials (e.g. an interpolant at Chebyshev points) 
whose convergence degenerates to algebraic  as $\epsilon \rightarrow 0$, as
the  convergence rate is dictated by the size of the Bernstein ellipse in the complex plane in which
$f(x)$ is analytic.  The final application is solving Schr\"odinger's equation when the potential is singular
or nearly singular. In particular, we consider the equation
$$
-0.1^2 u''(t) + V(t) u =\lambda u
$$
with $V(t) = \sqrt{t^2 + \epsilon^2} + (t-0.1)^2$ using Dirichlet conditions, showing that the
eigenstates can be calculated robsutly for  small $\epsilon$ by using a basis built out of
orthogonal polynomials on the hyperbola.

The paper is organized as follows. The next section is preliminary, where we fix notations and lay down 
the groundwork. Orthogonal structure on the quadratic curves will be discussed in the next five sections,
from Section 3 to Section 7, each standard quadratic curve is treated in its own section, in the order of 
(i) -- (v). In Section 8 we show that interpolation with this basis is feasible via suitably constructed quadrature rules.   Applications of the results are  investigated in Section 9.   
 
\section{Preliminary}
\setcounter{equation}{0}
 
Any second degree curve in the plane satisfies an equation of the form
$$
     \gamma(x,y):=  a_{1,1} x^2 + 2 a_{1,2} xy + a_{2,2} y^2 + 2 a_{1,3} x + 2 a_{2,3} y + a_{3,3} =0.
$$
Let $A = (a_{i,j})_{i,j =1}^3$ with $a_{j,i} := a_{i,j}$. The curve is called non-degenerate if $\det (A) \ne 0$ and
degenerate if $\det(A) = 0$. Let $\Delta:=  a_{1,1} a_{2,2} - a_{1,2}^2$. In the non-degenerate case, there are 
three classes of curves, ellipse ($\Delta > 0$), parabola ($\Delta= 0$), and hyperbola ($\Delta < 0$). In the 
degenerate case,  there are two classes, two intersecting real lines ($\Delta > 0$) and two parallel lines ($\Delta =0$).  
In the last case, we assume that the two parallel lines are distinct. If the two lines coincide, then an appropriate 
rotation and translation reduces the problem to the real line. Orthogonal polynomials on the real line have been studied extensively and are well-understood (cf. \cite{Sz}).

Let $\Pi_n^2$ denote the space of polynomials of degree at most $n$ in two real variables. Let $\gamma$ be 
a quadratic curve. With respect to the bilinear form \eqref{eq:ipd} defined on $\g$, a polynomial $P \in \Pi_n^2$ 
is called an orthogonal polynomial on the quadratic curve $\gamma$ if $\la P, Q\ra = 0$ for all polynomials 
$Q \in \Pi_m^2$ with $m < n$ and $\la P, P \ra > 0$. Let $\CH_n = \CH_n(\varpi) \subset \Pi_n^2$ be the space of orthogonal 
polynomials of degree $n$ on the quadratic curve.  

As mentioned in the introduction, we only need to consider a standard quadratic curve for each class. For each
quadratic curve $\gamma$ in (i) -- (v), the bilinear form \eqref{eq:ipd} defined on $\Omega_\g$ is an inner 
product on the space $\RR[x,y]/ \langle \gamma \rangle$ of polynomials modulo the ideal $\la \gamma\ra$. 
Using the explicit formula of $\gamma$ for each standard form, it is easy to see that the following
proposition holds. 

\begin{prop} 
The space $\CH_n$ has dimension $\dim \CH_0 =1$ and $\dim \CH_n = 2$ for $n \ge 1$. 
\end{prop}

An explicit basis for the space $\CH_n$ will be constructed for each class of quadratic curves in Sections 3--7. 
Let $\{Y_{n,1}, Y_{n,2}\}$ be an orthogonal basis of $\CH_n$. Then these two polynomials are orthogonal
to all polynomials of lower degrees and they are orthogonal to each other. Let $L^2(\varpi, \Omega)$ be the
$L^2$ space of Lebesgue integrable functions with finite norm $\|f\| = \la f, f\ra^{\f12}$. For 
$f \in L^2(\varpi, \Omega)$ its Fourier orthogonal series is defined by 
\begin{equation}\label{eq:Fourier}
     f =  \wh f_0 + \sum_{n=1}^\infty \left[ \wh f_{n,1} Y_{n,1} +   \wh f_{n,2} Y_{n,2} \right], 
\end{equation}
where 
$$
 \wh f_0:= \frac{ \la f,1\ra}{\la 1,1\ra} \quad \hbox{and}\quad 
     \wh f_{n,i}:= \frac{\la f, Y_{n,i} \ra}{\la Y_{n,i}, Y_{n,i} \ra}, \quad n \ge 1. 
$$
The partial sum operator $S_n f$ is defined by 
\begin{equation}\label{eq:Snf}
    S_n f :=  \wh f_0 + \sum_{k=1}^n \left[ \wh f_{k,1} Y_{k,1} +   \wh f_{k,2} Y_{k,2} \right]. 
\end{equation}
We shall study the convergence of the Fourier orthogonal series on quadratic curves. 

Our construction of orthogonal polynomials on the quadratic curves will heavilyuse  orthogonal polynomials
on the real line. Let $w$ be a nonnegative weight function on $\RR$. We let $p_n(w)$ denote an orthogonal 
polynomial of degree $n$ with respect to $w$ and let $h_n(w)$ denote the norm square of $p_n(w)$; more
precisely,  
$$
      \int_\RR  p_n(w;x) p_m(w; x) w(x) dx = h_n(w) \delta_{m,n}.
$$ 
Let $L^2(w)$ denote the $L^2$ space with respect to $w$ on $\RR$. The Fourier orthogonal series of 
$f \in L^2(w)$ is defined by 
\begin{equation} \label{eq:Fourier-series}
  f = \sum_{n=1}^\infty \wh f_n(w) p_n(w) \quad \hbox{with} \quad  \wh f_n(w) =
      \frac1{h_n(w)} \int_\RR f(y) p_n(w;y) w(y)dy. 
\end{equation}
The Parseval identity implies that 
$$
    \|f\|_{L^2(w,[0,1])}^2 = \sum_{n=0}^\infty \left|\wh f_n(w) \right|^2 h_n(w).
$$
The $n$-th partial sum of the Fourier orthogonal series with respect to $w$ is defined as 
\begin{equation} \label{eq:partial-sum}
  s_n(w;f)(x) := \wh f_0(w) p_0(w;x) +  \sum_{k=1}^n \wh f_k(w) p_k(w;x). 
\end{equation}

Our study of the Fourier orthogonal series on a quadratic curve consists of reducing the problem
to that of the Fourier orthogonal series on $\RR$. 

\section{Orthogonal polynomials on an ellipse}
\setcounter{equation}{0}

Each ellipse can be transformed to the unit circle under an affine transform. Hence, we only need to consider the
circle
$$
  \Omega = \{(x, y): x^2 + y^2 =1 \in \RR\} = \{(\cos \t, \sin \t): \t \in [0, 2\pi]\}.
$$ 
Orthogonal polynomials on circles are well studied. However, it should be pointed out that we are considering 
them as real polynomials in two variables, which are distinct from the thoroughly studied
Orthogonal Polynomials on the Unit Circle (OPUC)  \cite{BS}. The latter are analytic  polynomials in a 
complex $z$ variable and orthogonal with respect to an inner product defined on the unit circle in the 
complex plane. In our results, they are restrictions of homogeneous polynomials of two variables on the unit
circle in $\RR^2$. The result below is well--known, we include it here for completeness and also as an opening
to our study for other quadratic curves. 

Let $w$ be an even nonnegative weight function defined on $[-1,1]$. We consider the bilinear form defined by
\begin{align}
   \la f, g \ra &=  \int_{-\pi}^{\pi} f(\cos \t, \sin \t) g(\cos \t, \sin \t) w(\cos \t) d\t \notag\\
 &  = \int_{-1}^1 \Bigl[f(x,\sqrt{1-x^2}) g(x,\sqrt{1-x^2}) \notag\\
 &\qquad+
   	f(x,-\sqrt{1-x^2}) g(x,-\sqrt{1-x^2})  
\Bigr] {w(x) \over \sqrt{1-x^2}} dx \label{eq:ipd-circle}
\end{align}
which defines an inner product on the space $\RR[x,y]/\la x^2+y^2 -1\ra$. 

Let $\CP_n$ denote the space of homogeneous polynomials in two variables. Let $\CH_n(\varpi)$ denote the
space of orthogonal polynomials of degree $n$ with respect to this inner product.

\begin{thm}\label{thm:circlebasis}
Let $w(t)$ be an even weight function defined on $[-1,1]$. Define $w_{-\f12}(t) = (1-t^2)^{-\f12}w(t)$ and 
$w_{\f12}(t) = (1-t^2)^{\f12}w(t)$. The polynomials
\begin{equation}\label{eq:OPcircle}
     Y_{n,1} (x,y) =   p_n (w_{-\f12}; x) \quad \hbox{and}\quad Y_{n,2} (x,y) = y p_{n-1}(w_{\f12}; x)
\end{equation}
are homogeneous polynomials of degree $n$ in $(x,y)$ and they form an orthogonal basis of $\CH_n(\varpi)$
for all $n \ge 1$. Furthermore, 
\begin{equation}\label{eq:norm-circle}
  \la Y_{n,1}, Y_{n,1} \ra = 2 h_n(w_{-\f12}) \quad \hbox{and} \quad  \la Y_{n,2}, Y_{n,2} \ra = 2 h_{n-1}(w_{\f12}).
\end{equation}
\end{thm}

\begin{proof}
This result is known; see, for example, \cite[Section 4.2]{DX}. We give an outline of the proof for completeness. 
Since $w(t)$ is even, the polynomial $p_n (w; t)$ is an even polynomial if $n$ is even and an odd polynomial
if $n$ is odd. Consequently, it is easy to see that both $Y_{n,1}$ and $Y_{n,2}$ are homogeneous polynomials of
degree $n$ in $(x,y)$. Since $\sin \t$ is odd, it follows readily that $\la Y_{n,1}, Y_{m,2} \ra = 0$ for all $n \ge 0$ and
$m \ge 1$. Moreover, 
\begin{align*}
  \la Y_{n,1}, Y_{n,1} \ra & = 2 \int_{-1}^{1} \left[ p_n(w_{-\f12}; t)\right] ^2 \frac{w(t)}{\sqrt{1-t^2}} dt = 2 h_n(w_{-\f12})        
\end{align*}
and, similarly, $\la Y_{n,2}, Y_{n,2} \ra = 2 h_{n-1}(w_{\f12})$.
\end{proof}

The most well-known example are the spherical harmonics, for which $w(t) =1$. The orthogonal polynomials are
\begin{align*}
  Y_{n,1}(x,y) & = \cos n \t =  T_n(x), \\
   Y_{n,2}(x,y) & = \sin n \t = 
      yU_{n-1}(x), 
\end{align*}
where $T_n(t)$ and $U_{n}(t)$ are the Chebyshev polynomials of the first and 
the second kind, respectively. 
The restrictions of $Y_{n,1}$ and $Y_{n,2}$ to the circle are the eigenfunctions
of $d^2/d \t^2$ (which is the Laplace--Beltrami operator on the circle) in the sense that they satisfy $u''(\t) = -n u (\t)$. 

Another family of examples are $h$-harmonics associated with the dihedral group $I_k$ of $k$-regular polygons 
in the plane (cf. \cite[Section 7.6]{DX}). For $k=1,2,\ldots$, the weight function in \eqref{eq:ipd-circle} is 
defined by 
$$
    w(\cos \t) = |\sin (k\t )|^{2 \a} |\cos (k \t)|^{2\b}, \qquad \a, \b  \ge 0.
$$
For each $k$, the corresponding polynomials $Y_{n,1}$ and $Y_{n,2}$ are the eigenfunctions of a second 
order differential-difference operator on the circle \cite[Section 7.6]{DX}.

We can connect the Fourier orthogonal series on the circle to the Fourier orthogonal series in terms of 
$w_{-\f12}$ and $w_{\f12}$ on the interval $[-1,1]$. Recall the definition of partial sum operators defined in 
\eqref{eq:Snf} and \eqref{eq:partial-sum}. 

\begin{thm} \label{thm:Snf-circle}
Let $f$ be a function defined on $\Omega$. Define 
$$
 f_e(x) := \frac{f(x,  \sqrt{1-x^2}) + f(x, -  \sqrt{1-x^2})}{2} 
 $$
 and
 $$
 f_o(x) := \frac{f(x, \sqrt{1-x^2})-f(x,- \sqrt{1-x^2})}{2 \sqrt{1-x^2}}.
$$
 Then 
\begin{align} \label{eq:Snf-circle}
  S_n f(x,y) \,= s_n(w_{-\f12}; f_e, x) + y s_{n-1} (w_{\f12}; f_o, x). 
 \end{align}
In particular, if $s_n(w_{-\f12}; f_e, x) \to f_e(x)$ and $s_{n-1} (w_{\f12}; f_o, x) \to f_o(x)$ as $n\to \infty$, then 
$S_n f(x,y) \to f(x,y)$ as $n \to \infty$. Furthermore, if $f\in L^2(\varpi, \Omega)$, then 
\begin{align}\label{eq:SnfL2-circle}
\|f - S_n(f) \|_{L^2(\varpi,\Omega)}^2 = &\,  \|f_e - s_{n} (w_{-\f12};f_e)  \|_{L^2(w_{-\f12})}^2   + \|f_o - s_{n-1} (w_{\f12};f_o) \|_{L^2(w_{\f12})}^2.  
\end{align}
\end{thm}

\begin{proof}
By the expression of $Y_{n,1}$ in \eqref{eq:OPcircle}, 
\begin{align*}
  \la f, Y_{n,1} \ra &\, = \int_{-\pi}^\pi f(\cos \t,\sin \t) p_n(w_{-\f12}; \cos \t) w(\cos \t) d\t  \\
       & \,= \int_{0}^\pi \left[f(\cos \t,\sin \t) + f(\cos \t, -\sin \t)\right] p_n(w_{-\f12}; \cos \t) w(\cos \t) d\t \\
       & \,=   2 \int_{-1}^1 f_e (t) p_n(w_{-\f12}; t) w_{-\f12}(t) dt  =  2 \la f_e, p_n(w_{-\f12}) \ra_{L^2(w_{-\f12})}.
\end{align*}
Similarly, by the expression of $Y_{n,2}$ in \eqref{eq:OPcircle}, 
\begin{align*}
  \la f, Y_{n,2} \ra &\, =  \int_0^\pi \left[f(\cos \t,\sin \t) - f(\cos \t, -\sin \t)\right] \sin \t p_{n-1} (w_{\f12}; \cos \t) w(\cos \t) d\t \\
      & \,=   \int_0^\pi \frac{f(\cos \t,\sin \t) - f(\cos \t, -\sin \t)}{\sin \t} \sin^2 \t p_{n-1}(w_{\f12}; \cos \t) w(\cos \t) d\t \\
      & \,=   2 \int_{-1}^1 f_o (t) p_{n-1}(w_{\f12}; t) w_{\f12}(t) dt  =  2 \la f_o, p_{n-1}(w_{\f12}) \ra_{L^2(w_{\f12})}.
\end{align*}
Consequently, by \eqref{eq:norm-circle}, we obtain
\begin{align*}
  \wh f_{n,1} =  \frac{ \la f_e, p_n(w_{-\f12}) \ra_{L^2(w_{-\f12})}}{ h_n(w_{-\f12})} =  \wh {f_e}_n (w_{-\f12}), \,\,
  \wh f_{n,2} =  \frac{ \la f_o, p_{n-1}(w_{\f12}) \ra_{L^2(w_{\f12})}}{ h_{n-1}(w_{\f12})} =  \wh {f_o}_{n-1} (w_{\f12}) 
\end{align*}
in the notation of \eqref{eq:Fourier-series}. Hence, \eqref{eq:Snf-circle} follows readily from the definition 
of the partial sum operators in \eqref{eq:Snf}. 

With $y = \sqrt{1-x^2}$, we see that $f(x,y) = f_e(x) + y f_o(x)$. It follows that 
\begin{align*}
  f(x, y) - S_n f(x,y) = f_e (x) -  s_n(w_{-\f12}; f_e, x) 
    + y \left(f_o(x) - s_{n-1} (w_{\f12}; f_o, x)\right).
\end{align*}
Consequently, the convergence of $S_n f$ follows form that of $ s_n(w_{-\f12};f_e)$ and $s_{n-1}(w_{\f12}; f_o)$. 
Finally, let $F = f_e -s_n(w_{-\f12}; f_e)$ and $G =f_o- s_{n-1} (w_{\f12}; f_o)$. Then the last displayed identity 
shows that
\begin{align*}
  \|f-S_n f\|_{L^2(\varpi,\Omega)}^2  = & \, \int_{-\pi}^\pi |F(\cos \t) + \sin \t G(\cos \t)|^2 w(\cos \t) d\t \\
     =  & \,  \int_{-\pi}^\pi \left( |F(\cos \t)|^2+ |\sin \t G(\cos \t)|^2 \right) w(\cos \t) d\t \\
     =  & \int_{-1}^1 |F(t)|^2 w_{-\f12}(t) dt + \int_{-1}^1 |G(t)|^2 w_{\f12}(t) dt,
\end{align*}
where we have used the fact that the integral of $\sin \t F(\cos \t) G(\cos \theta)$ is zero in the second step. 
This gives \eqref{eq:SnfL2-circle} and completes the proof. 
\end{proof}

This theorem is also known but probably not stated in this precise form. When $f$ is chosen as 
$f(x,y) = g(\cos \t)$, then the Fourier orthogonal series of $f$ becomes the Fourier orthogonal 
series of $g$ in $L^2(w; [-1,1])$. For the classical Fourier series, it is the familiar fact that the Fourier 
series of $f(\cos\t)$ is the Fourier cosine series of $f(t)$. We choose to give a complete proof here since 
it is the harbinger of what will come for other quadratic curves. 

\section{Orthogonal polynomials on a parabola}
\setcounter{equation}{0}

Under an affine transform, we only need to consider the parabola defined by 
$$
    y = x^2 \qquad \hbox{or} \qquad   \Omega = \{(x, y): y = x^2, \, x \in \RR\},
$$ 
and consider even weights $\varpi(x,y) = \varpi(-x,y)$. Since $y = x^2$, we can write $\varpi(x,y) = w(x^2)$ for
a weight function $w$ on $\RR_+ = [0,+\infty)$, so that the bilinear form becomes
\begin{equation}\label{eq:ipd-parabola}
 \la f, g \ra = \int_\RR f(x,x^2) g(x,x^2) w(x^2) dx, 
\end{equation}
which defines an inner product on the space $\RR[x,y] /\la y - x^2 \ra$.

\subsection{Orthogonal polynomials} 
Let $\CH_n(\varpi)$ denote the space of orthogonal polynomials of degree $n$ with respect to the inner product
\eqref{eq:ipd-parabola}. 

\begin{thm}\label{thm:parabolabasis}
Let $w$ be a weight function defined on $\RR_+$. Define $w_{-\f12}(t) := t^{-\f12} w(t)$ and $w_\f12(t) := t^{\f12} w(t)$. 
Then the polynomials 
\begin{equation}\label{eq:OP-parabola}
     Y_{n,1} (x,y) = p_n (w_{-\f12}; y) \quad \hbox{and}\quad Y_{n,2} (x,y) = x p_{n-1}(w_\f12; y)
\end{equation}
form an orthogonal basis of $\CH_n(\varpi)$ for $n \ge 1$. Furthermore,
\begin{equation} \label{eq:norm-parabola}
  \la Y_{n,1}, Y_{n,1} \ra = h_n(w_{-\f12}) \quad \hbox{and} \quad  \la Y_{n,2}, Y_{n,2} \ra = h_{n-1} (w_{\f12}).
\end{equation}
\end{thm}

\begin{proof}
For $n \ge 0$ and $m \ge 0$, 
\begin{align*}
 \la Y_{n,1},  Y_{m,2} \ra = \, & \int_\RR Y_{n,1} (x,x^2) Y_{m,2}(x,x^2)  w(x^2) dx  \\
      = \, &  \int_\RR x p_n (w_{-\f12}, x^2) p_{n-1}(w_\f12; x^2) w(x^2) dx  =0
\end{align*}
since the integrand is odd. For $n, m \in \NN_0$, 
\begin{align*}
\la Y_{n,1},  Y_{m,1} \ra = \, &  \int_\RR p_n (w_{-\f12}; x^2) p_{m}(w_{-\f12}; x^2) w(x^2) dx   \\
   = \, & 2 \int_0^\infty p_n (w_{-\f12}; x^2) p_{m}(w_{-\f12}; x^2) w(x^2) dx   \\
   =  \, & \int_0^\infty p_n (w_{-\f12}; t) p_{m}(w_{-\f12}; t) t^{-\f12} w(t) dt = h_n(w_{-\f12}) \delta_{n,m}.
\end{align*}
Similarly, for $n, m \in \NN$, 
\begin{align*}
\la Y_{n,2},  Y_{m,2} \ra = \, &  \int_\RR x^2 p_{n-1} (w_{\f12}; x^2) p_{m-1}(w_{\f12}; x^2) w(x^2) dx   \\
    =  \, & \int_0^\infty p_{n-1} (w_{\f12}; t) p_{m-1}(w_{\f12}; t) t^{\f12} w(t) dt =h_{n-1}(w_{\f12}) \delta_{n,m}.
\end{align*}
Since it is evident that $Y_{n,1}$ and $Y_{n,2}$ are polynomials of degree $n$ in two variables, we see that
they form an orthogonal basis for $\CH_n(\varpi)$. 
\end{proof}

We consider two examples of classical orthogonal polynomials. Our first example is on the entire parabola 
$\{(x,y): y=x^2, x \in \RR\}$ with the Hermite weight $w(t) = e^{-t^2}$. %\sotodo{added one sentence.} 
Let $L_n^\a$ denote the Laguerre polynomials, which are orthogonal with respect to $x^\a e^{-x}$ on $\RR_+$. 

\begin{prop}
For the inner product with respect to the Hermite weight
$$
  \la f, g \ra = \int_{\RR} f(x,x^2) g(x,x^2) e^{-x^2} dx  
$$
and $n \ge 1$, an orthogonal basis of $\CH_n$ is given by the polynomials 
$$
  Y_{n,1} (x,y) = L_n^{-\f12}(y) \quad \hbox{and} \quad Y_{n,2}(x,y) = x L_{n-1}^{\f12}(y).
$$
Furthermore, they are the eigenfunctions of a second order differential operator; more precisely, 
$Y_{n,1}$ and $Y_{n,2}$ satisfy the equation
\begin{equation} \label{eq:PDE-Hermite}
  \CL u : =  x \partial_{x y} u + y \partial_{yy}u+ (\tfrac12 - y) \partial_y u - x \partial_x u = -n u. 
\end{equation}
\end{prop}

\begin{proof}
We only need to verify the differential equations. The Laguerre polynomial $v(y)= L_n^\alpha(y)$ satisfies
the equation
$$
       v '' + (\a +1 - y) v' = -n v.
$$
For $u = Y_{n,1}$, we have $\partial_x u =0$ and it is easy to see that \eqref{eq:PDE-Hermite} becomes the 
differential equation for $v(y) = L_n^{-\f12}(y)$. For $u = Y_{n,2}$, we write $u(x,y) = x v(y)$ with 
$v(y)  =  L_{n-1}^{\f12}(y)$. Then
\begin{align*}
  \CL u = x v' + x y v'' + (\tfrac12 - y) x v' - x v & = x (y v'' +  (\tfrac32 - y) v') - x v \\
    &   = x  ( -(n-1) v -v )  = -n x v = -n u.
\end{align*}
This completes the proof. 
\end{proof}

Our second example is over the parabola with the Gegenbauer weight $w(t) = (1-t^2)^{\l-\f12} \chi_{[-1,1]}(t)$, 
where $\chi_E$ denotes the characteristic function of the set $E$. The support set of $w$ is finite, so that 
our orthogonal polynomials are over a finite parabola $\{(x,y): y=x^2, x \in [-1,1]\}$. %\sotodo{added two sentences.}
Let $P_n^{(\a,\b)}$ denote the Jacobi polynomials, which are orthogonal with respect to the weight function
$(1-x)^\a (1+x)^\b$ on the interval $[-1,1]$.

\begin{prop}
For the inner product with respect to the Gegenbauer weight
$$
  \la f, g \ra = \int_{-1}^1 f(x,x^2) g(x,x^2) (1-x^2)^{\l -\f12}dx  
$$
and $n \ge 1$, an orthogonal basis of $\CH_n$ is given by the polynomials 
$$
  Y_{n,1} (x,y) = P_n^{(\l-\f12, -\f12)}(2y-1) \quad \hbox{and} \quad Y_{n,2}(x,y) = x P_{n-1}^{(\l-\f12, \f12)}(2y-1).
$$
\end{prop}

\begin{proof}
Since the polynomial $P_n^{(\l-\f12, -\f12)}(2t-1)$ is orthogonal with respect to $t^{-\f12}(1-t)^{\l-\f12}$ on $[0,1]$ 
and the polynomial $P_n^{(\l-\f12, \f12)}(2t-1)$ is orthogonal with respect to $t^{\f12}(1-t)^{\l-\f12}$ on $[0,1]$, we
see that $Y_{n,1}$ and $Y_{n,2}$ form a basis of $\CH_n$. 
\end{proof}

It is known that the Jacobi polynomial $v(t) = P_n^{(\a,\b)}(2t-1)$ satisfies the differential equation
$$
       t(1-t) v'' + ( \b+1 - (\a+\b+2) t) v' = -n (n +\a+\b+1) v.
$$
However, since the eigenvalues are quadratic in $n$, we cannot combine terms as in the proof of 
\eqref{eq:PDE-Hermite} when working with $Y_{n,2}$ and, as a consequence,  %\sotodo{added a sentence}
we are not able to deduce a second order differential operator that
has $Y_{n,1}$ and $Y_{n,2}$ as eigenfunctions with the same eigenvalue. 

\subsection{Fourier orthogonal series}

Let $L^2(\varpi, \Omega)$ be the $L^2$ space on $\Omega$ with the norm $\|\cdot\|_{L^2(\varpi,\Omega)}$ defined by
$$
  \|f\|_{L^2(\varpi, \Omega)}^2: = \la f,f\ra =  \int_\RR |f(x,x^2)|^2 w(x^2)dx.   
$$
For $f \in L^2(\varpi, \Omega)$, we consider the Fourier orthogonal series with respect to the basis 
\eqref{eq:OP-parabola} as defined in \eqref{eq:Fourier}.

\begin{thm} \label{thm:Snf-parabola}  %\sotodo{notation changed, switching from $x$ to $y$ and eliminate $|phi$}
Let $f$ be a function defined on $\Omega$. Define 
$$
 f_e(y) := \frac{f(\sqrt{y},y) + f(-\sqrt{y},y)}{2} \quad \hbox{and} \quad f_o(y) := \frac{f(\sqrt{y},y)-f(-\sqrt{y},y)}{2 \sqrt{y}}. 
$$
Then
\begin{align} \label{eq:Snf-parabola}
  S_n f(x,y) \,= s_n(w_{-\f12}; f_e, y) + x s_{n-1} (w_{\f12}; f_o, y). 
 \end{align}
In particular, if $s_n(w_{-\f12}; f_e) \to f_e$ and $s_{n-1} (w_{\f12}; f_o) \to f_o$ as $n\to \infty$, 
then $S_n f(x,y) \to f(x,y)$ as $n \to \infty$. Furthermore, if $f\in L^2(\varpi, \Omega)$, then 
\begin{align*}%\label{eq:SnL2-parabola}
\|f - S_n(f) \|_{L^2(\varpi,\Omega)}^2 =  \|s_{n} (w_{-\f12} ;f_e) - f_e\|_{L^2(w_{-\f12} )}^2   + \|s_{n-1} (w_{\f12};f_o) - f_o\|_{L^2(w_{\f12} )}^2. %\notag 
\end{align*}
\end{thm}

\begin{proof}
The proof is similar to that of Theorem \ref{thm:Snf-circle}. We shall be brief. 
By the expression of $Y_{n,1}$ in \eqref{eq:OP-parabola}, 
\begin{align*}
  \la f, Y_{n,1} \ra &\, = \int_\RR f(x,x^2) p_n(w_{-\f12}; x^2) w(x^2) dx  \\
       & \,= \int_0^\infty \left[f(x,x^2)+ f(-x,x^2) \right] p_n(w_{-\f12}; x^2) w(x^2) dx \\
       & \,=   \int_0^\infty f_e (y) p_n(w_{-\f12}; y) y^{-\f12} w(y) dy 
         =  \la f_e, p_n(w_{-\f12})\ra_{L^2(w_{-\f12})}.
\end{align*}
Similarly, by the expression of $Y_{n,2}$ in \eqref{eq:OP-parabola}, we can derive 
\begin{align*}
  \la f, Y_{n,2} \ra 
 %    &\, = \int_\RR f(x,x^2) x p_{n-1}(w_{\f12}; x^2) w(x^2) dx  \\
  %     & \,= \int_0^\infty \left[f(x,x^2)- f(-x,x^2) \right] x p_{n-1} (w_{\f12}; x^2) w(x^2) dx \\
   %    & \,=   \int_0^\infty (f_o\circ \phi) (t) p_{n-1}(w_{\f12}; t) t^{\f12} w(t) dt 
         =  \la f_o, p_{n-1}(w_{\f12})\ra_{L^2(w_{\f12})}.
\end{align*}
Consequently, by \eqref{eq:norm-parabola}, we obtain
$$
  \wh f_{n,1} =  \wh {(f_e)}_n (w_{-\f12}) \quad\hbox{and}\quad \wh f_{n,2}  =   \wh {(f_o)}_{n-1} (w_{\f12})
$$
in the notation of \eqref{eq:Fourier-series}, from which \eqref{eq:Snf-parabola} follows. 

The rest of the proof follows from $f(x,x^2) = f_e(x^2) + x f_o(x^2)$, which shows that 
\begin{align*}
  f(x, x^2) - S_n f(x,x^2) = f_e (x^2) -  s_n(w_{-\f12}; f_e, x^2) 
    + x\left(f_o(x^2) - s_{n-1} (w_{\f12}; f_o, x^2)\right)
\end{align*}
and its two terms in the right-hand side have different parity, since the two functions in the brackets 
are obviously even functions.  
%\begin{align*}
%  \|f-S_n f\|_{L^2(\varpi,\Omega)}^2  = & \, \int_\RR  |F(x^2) + x G(x^2)|^2 w(x^2) dx \\
%     =  & \, \int_\RR \left( |F(x^2)|^2 + x^2 |G(x)|^2 \right) w(x^2) dx \\ 
 %    = & \int_0^\infty  |F(t)|^2 t^{-\f12} w(t) dt + \int_0^\infty|G(t)|^2 t^{\f12} w(t) dt,
%\end{align*}
%where we have used the fact that the integral of $x F(x^2) G(x^2)$ is zero in the second step. This gives 
%\eqref{eq:SnL2-parabola} and completes the proof. 
\end{proof}

\section{Orthogonal polynomials on a hyperbola curve}
\setcounter{equation}{0}

Under an affine transform, we only need to consider the hyperbola defined by 
$$
     x^2 - y^2 =1 \qquad \hbox{or} \qquad   \Omega = \{(x, y): x^2- y^2 =1\}.
$$ 
The resulted quadratic curves has two branches. We consider two cases. 

\subsection{$\Omega$ with two branches} 
We consider the bilinear form defined for both branches of the hyperbola $x^2 = y^2 +1$ and consider 
$$
\Omega = \left\{ \left(\sqrt{ y^2 +1}, y\right): y \in \RR \right\} \bigcup  \left\{- \left(\sqrt{ y^2 +1}, y\right): y \in \RR \right\}.
$$
For weights satisfying $\varpi(x,y) = \varpi(-x,y)$, even in $x$, we can use $x^2 = y^2 + 1$ and write 
$\varpi(x,y) = w(y)$ for some $w$ defined on $\RR$. We then consider the bilinear form
\begin{align}\label{eq:ipd-hyper}
  \la f, g\ra =   \int_{-\infty}^\infty  & \left[  f(\sqrt{y^2+1},y) g(\sqrt{y^2+1},y) \right. \\
    & \qquad  + \left.      f(- \sqrt{y^2+1},y) g(- \sqrt{y^2+1},y) \right] w(y) dy, \notag
\end{align}
which is an inner product on $\RR[x,y]/\la x^2 - y^2 -1\ra$. Again let $\CH_n(\varpi)$ be the space of orthogonal
polynomials of degree $n$ with respect to the inner product $\la \cdot, \cdot \ra$. 

\begin{thm}\label{thm:TwoBranch}
Let $w$ be a weight function on $\RR$ and let $w_1(t) := (1+t^2)w(t)$. Then the polynomials 
\begin{equation}\label{eq:OP-hyper}
     Y_{n,1} (x,y) = p_n (w; y) \quad \hbox{and}\quad \qquad Y_{n,2} (x,y) = x p_{n-1}(w_1; y)
\end{equation}
form an orthogonal basis of $\CH_n(\varpi)$ for $n \ge 1$. Furthermore,
\begin{equation}\label{eq:norm-hyper}
  \la Y_{n,1}, Y_{n,1} \ra = 2 h_n(w) \quad \hbox{and} \quad  \la Y_{n,2}, Y_{n,2} \ra = 2 h_{n-1}(w_1).
\end{equation}
\end{thm}

\begin{proof}
For $n \ge 0$ and $m \ge 1$, we have
$$
  \la Y_{n,1}, Y_{m,2} \ra =  \int_{-\infty}^\infty  p_n(w; y) p_{m-1}(w_1;y) 
      \left[ \sqrt{y^2+1}  -  \sqrt{y^2+1} \right]w(y) dy = 0.
$$
Furthermore, for $n, m \in \NN_0$,   
$$
  \la Y_{n,1}, Y_{m,1} \ra =  2 \int_{-\infty}^\infty  p_n(w; y) p_{m}(w;y) w(y) dy = 2 h_n(w) \delta_{n,m} 
$$
and, for $n, m \in \NN$,
$$
  \la Y_{n,2}, Y_{m,2} \ra =  2 \int_{-\infty}^\infty  p_{n-1}(w_1; y) p_{m-1}(w_1;y)(1+y^2)w(y) dy = 
     2 h_{n-1}(w_1) \delta_{n,m} 
$$
since $w_1 = (1+\{\cdot\}^2) w$. This completes the proof.
\end{proof}

If $w$ is an even function on $\RR$, the orthogonal polynomial $p_{n-1}(w_1; \cdot)$ can be constructed 
explicitly as shown in the following proposition. 

\begin{prop}
Let $w_1(t) := (1+t^2)w(t^2)$  %\sotodo{corrected} be an even weight function on $\RR$. We further
define $w_{-\f12}(t) = t^{- \f12} w(t)$ and $w_{\f12}(t) = t^{\f12} w(t)$. Then 
\begin{align*}
p_{2n}(w_1; t) & =  \frac{1}{1+t^2} \left(p_{n+1}(w_{-\f12}; t^2) p_n(w_{-\f12};-1)- p_{n+1}(w_{-\f12}; -1) p_n(w_{-\f12}; t^2)\right), \\
p_{2n+1}(w_1; t) & =  \frac{t}{1+t^2} \left(p_{n+1}(w_{\f12}; t^2) p_n(w_{\f12};-1)- p_{n+1}(w_{\f12}; -1) p_n(w_{\f12}; t^2) \right).
\end{align*}
\end{prop}

\begin{proof}
Changing variables $s = t^2$, we obtain
$$
   \int_\RR f(t^2) g(t^2) w_1(t) dt = \int_0^\infty f(s) g(s) (1+s) s^{-\f12} w(s) ds, 
$$
from which it is easy to see that 
$$
  p_{2n}(w_1; t) = p_n ( (1+ \{\cdot\} ) w_{-\f12}; t^2) \quad \hbox{and}\quad  
  p_{2n+1}(w_1; t) = t p_n ( (1+\{\cdot\}) w_{\f12}; t^2). 
$$
The orthogonal polynomials for $(1+t)  w_{-\f12}$ and $(1+t) w_{\f12}$ can be expressed by those for $w_{-\f12}$ and 
$w_{\f12}$, respectively, by the method of Christoffel \cite[Theorem 2.5]{Sz}, which gives the stated formulas. 
\end{proof}

Let $L^2(\varpi, \Omega)$ be the $L^2$ space on $\Omega$ with the norm $\|\cdot\|_{L^2(\varpi,\Omega)}$ defined by
\begin{align*}
  \|f\|_{L^2(\varpi, \Omega)}^2: =\la f,f\ra 
      = \int_{-\infty}^\infty \left[  \left |f(\sqrt{y^2+1},y)\right |^2 + \left |f(-\sqrt{y^2+1},y)\right |^2 \right]  w(y)dy.   
\end{align*}
We consider the Fourier orthogonal expansions as defined in \eqref{eq:Fourier}. 
 
\begin{thm} \label{thm:Snf-hyper}
Let $f$ be a function defined on $\Omega$. Define 
$$
 f_e(y) := \frac{f(\sqrt{y^2+1},y) + f(-\sqrt{y^2+1},y)}{2}
 $$
 and
 $$
  f_o(y) := 
     \frac{f(\sqrt{y^2+1},y) - f(-\sqrt{y^2+1},y) }{2 \sqrt{y^2+1}}. 
$$
Then the partial sum operator for the Fourier series with respect to \eqref{eq:ipd-hyper} satisfies
\begin{align} \label{eq:Snf-hyper}
  S_n f(x,y) \,= s_n(w; f_e, y) + x s_{n-1} (w_1; f_o, y). 
 \end{align}
In particular, if $s_n(w;f_e,x) \to f_e(x)$ and $s_n(w_1;f_o,x) \to f_o(x)$ as $n\to \infty$, then 
$S_n f(x,y) \to f(x,y)$ as $n \to \infty$. Furthermore, if $f\in L^2(\Omega, w)$, then 
\begin{align}\label{eq:SnL2-hyper}
\|f - S_n(f) \|_{L^2(\varpi,\Omega)}^2 = \|s_{n-1} (w ;f_e) - f_e\|_{L^2(w)}^2 + \|s_{n-1} (w_1;f_o) - f_o \|_{L^2(w_1)}^2.  
\end{align}
\end{thm}

\begin{proof}
The proof is again similar to that of Theorem \ref{thm:Snf-circle}. We shall be brief. 
Since $Y_{n,1}$ in \eqref{eq:OP-hyper} is independent of $x$,  
\begin{align*}
  \la f, Y_{n,1} \ra &\, = \int_{-\infty}^\infty \left[ f(\sqrt{y^2+1},y) + f(-\sqrt{y^2+1},y)\right] p_n(w; y) w(y) dy  =  \la f_e, p_n(w)\ra_{L^2(w)}.
\end{align*}
Similarly, by the expression of $Y_{n,2}$ in \eqref{eq:OP-hyper}, 
\begin{align*}
  \la f, Y_{n,2} \ra &\, = \int_{-\infty}^\infty \left[ f(\sqrt{y^2+1},y) - f(-\sqrt{y^2+1},y)\right] 
            p_{n-1}(w_1; y) \sqrt{y^2+1} w(y) dy\\
       & \,= \int_0^\infty f_0(y) p_{n-1} (w_1; y) w_1(y) dy 
         = \la f_o, p_{n-1}(w_1)\ra_{L^2(w_1)}.
\end{align*}
Consequently, by \eqref{eq:norm-hyper}, we obtain
$$
  \wh f_{n,1} =  \frac{ \la f_e, p_n(w)\ra_{L^2(w)} }{ h_n(w)} = \wh {f_e}_n(w), \qquad  
  \wh f_{n,2} =  \frac{  \la f_o, p_{n-1}(w_1)\ra_{L^2(w_1)} }{ h_{n-1}(w_1)} =  \wh {f_o}_n(w_1),  
$$
from which \eqref{eq:Snf-hyper} follows. The rest of the proof follows as that of 
Theorem \ref{thm:Snf-circle}. 
\end{proof}

\subsection{$\Omega$ with one branch} 
Here we consider only one brach of the hyperbola and we define
$$
\Omega = \{(x,y): x^2 - y^2 =1, \, x \ge 1\}.
$$
Comparing with \eqref{eq:ipd-hyper}, we can define the bilinear form on $\Omega$ as 
\begin{align} \label{eq:ipd-hyper-half}
  \la f, g\ra :=  \int_\RR f(\sqrt{y^2+1}, y) g(\sqrt{y^2+1}, y)  w(y)dy.
\end{align}

In order to construct an orthogonal basis, we  parametrize the integral in the $x$ variable instead of the $y$ variable. 
A change of variable shows that 
\begin{align*}
 \int_{- \infty}^\infty f(\sqrt{1+y^2}, y) dy & = \int_{0}^{ \infty} f(\sqrt{1+y^2}, - y) dy+\int_{0}^\infty f(\sqrt{1+y^2}, y) dy \\
   & = \int_{1}^\infty  \left[f(x, - \sqrt{x^2-1}) dy+ f(x, \sqrt{x^2-1})\right] \frac{x}{\sqrt{x^2-1}}dx.
\end{align*}
Accordingly, for a weight function $w_0$ defined on $[1,\infty)$, we define the bilinear from
\begin{align} \label{eq:ipd-hyper2}
  \la f, g\ra :=  \int_1^\infty  & \left[ f(x, \sqrt{x^2-1}) g(x, \sqrt{x^2-1})  \right.   \\
        & \qquad \left. + f(x, - \sqrt{x^2-1}) g(x, -\sqrt{x^2-1}) \right]w_0(x) dx. \notag
\end{align}

\begin{thm}\label{thm:OneBranch}
Let $w_0$ be a weight function defined on $[1,\infty)$ and let $w_1(t):= (t^2-1) w_0(t)$. Then the polynomials 
\begin{align} \label{eq:OP-hyper2}
     Y_{n,1} (x,y) = p_n(w_0; x) \quad \hbox{and}\quad Y_{n,2} (x,y) = y p_{n-1}(w_1; x)
\end{align}
form an orthogonal basis of $\CH_n(\varpi)$ for $n \ge 1$. Furthermore,
$$
  \la Y_{n,1}, Y_{n,1} \ra = 2 h_n(w_0) \quad \hbox{and} \quad  \la Y_{n,2}, Y_{n,2} \ra = 2 h_{n-1}(w_1).
$$
\end{thm}
 
\begin{proof}
For $n \ge 0$ and $m \ge 0$, we have
$$
  \la Y_{n,1}, Y_{m,2} \ra =  2 \int_{1}^\infty  p_n(w_0; x) p_{m-1}(w_1;x) \left[ \sqrt{x^2-1}  -  \sqrt{x^2-1} \right] w_0(x) dx =0.
$$
For $n, m \in \NN_0$, we have 
$$
  \la Y_{n,1}, Y_{m,1} \ra =  2 \int_{1}^\infty  p_n(w_0; x) p_{m}(w_0;x) w_0(x) dx = 2 h_n(w_0) \delta_{n,m} 
$$
and 
$$
  \la Y_{n,2}, Y_{m,2} \ra =  2 \int_{1}^\infty  p_{n-1}(w_1; x) p_{m-1}(w_1;x)(x^2-1)w_0(x) dx = 2 h_{n-1}(w_1)
    \delta_{n,m} 
$$
since $w_1 = (\{\cdot\}^2-1) w_0$. This completes the proof.
\end{proof}
 
\begin{cor}
Let $w$ be and even weight function on $\RR$ and let 
$$
  w_0(x) = \frac{x w(\sqrt{x^2-1})}{\sqrt{x^2-1}}, \qquad x \ge 1.
$$ 
Then the polynomials in \eqref{eq:OP-hyper2} form an orthogonal basis of $\CH_n(\varpi)$ with
respect to the inner product \eqref{eq:ipd-hyper-half}.
\end{cor}

Let $L^2(\varpi, \Omega)$ be the $L^2$ space on $\Omega$ with the norm $\|\cdot\|_{L^2(\varpi,\Omega)}$ defined by 
\begin{align*}
  \|f\|_{L^2(\varpi, \Omega)}^2: =\la f,f\ra 
      = \int_1^\infty  \left[ \left | f(x, \sqrt{x^2-1}) \right|^2 + \left|f(x, - \sqrt{x^2-1})\right |^2 \right]w_0(x) dx.
\end{align*}
We consider the Fourier orthogonal expansions as in \eqref{eq:Fourier}. 
 
\begin{thm} \label{thm:OneBranch2}
Let $f$ be a function defined on $\Omega$. Define 
$$
 f_e(x) := \frac{f(x,y) + f(x,-y)}{2} \quad \hbox{and} \quad f_o(y) := 
     \frac{f(x,y) - f(x,-y) }{2 y}, \quad y = \sqrt{x^2 -1}. 
$$
Then the partial sum operator for the Fourier series with respect to \eqref{eq:ipd-hyper2} satisfies
\begin{align} \label{eq:Snf-hyper2}
  S_n f(x,y) \,= s_n(w_0; f_e, x) + y s_{n-1} (w_1; f_o, x). 
 \end{align}
In particular, if $s_n(w_0;f_e,x) \to f_e(x)$ and $s_n(w_1;f_o,x) \to f_o(x)$ as $n\to \infty$, then $S_n f(x,y)$ 
converges to $f(x,y)$ as $n \to \infty$. Furthermore, 
if $f\in L^2(\Omega, w)$, then 
\begin{align}\label{eq:SnL2-hyper2}
\|f - S_n(f) \|_{L^2(\varpi,\Omega)}^2 = \|s_{n-1} (w_0 ;f_e) - f_e\|_{L^2(w_0)}^2   + \|s_{n-1} (w_1;f_o) - f_o \|_{L^2(w_1)}^2.  
\end{align}
\end{thm}
 
The proof is similar to that of Theorem \ref{thm:Snf-hyper} and we skip it.

\section{Two intersecting lines}
\setcounter{equation}{0}

Two intersecting lines can be transformed to the two coordinate axes by an affine transform. As a result, we
consider 
$$
  \Omega = \{(x, y): x = 0 \, \, \hbox{or} \, \, y =0\}
$$ 
Let $w_1$ and $w_2$ be two weight functions defined on the real line. Define a bilinear form
\begin{equation} \label{eq:ipd-lines}
    \la f,g\ra_{w_1,w_2} = \int_\RR f(x,0) g(x,0) w_1(x) dx +  \int_\RR f(0,y) g(0,y) w_2(y) dy, 
\end{equation}
which is an inner product for $\RR[x,y]/\la x y \ra$. We will make use of the results in \cite{OX}, which 
studies orthogonal polynomials on a wedge, defined as two line segments sharing a common end point. 
The standard wedge chosen in \cite{OX} is 
$$
  \Omega_{\rm wdg}:= \{(x,1): 0 \le x \le 1\} \cup \{(1,y): 0 \le y \le 1\}
$$ 
and the bilinear form there is defined as 
$$
    \la f,g\ra_{w_1,w_2}^{\rm wdg} = \int_0^1 f(x,1) g(x,1) w_1(x) dx +  \int_0^1 f(1,y) g(1,y) w_2(y) dy. 
$$
A simple transform $(x,y) \mapsto (1-x,1-y)$ shows that we could consider the wedge as 
$\{(x, 0): 0 \le x \le 1\} \cup \{(0, y): 0 \le y \le 1\}$. The corresponding bilinear form is then a special
case of \eqref{eq:ipd-lines} with respect to weight functions $w_1(x) \chi_{[0,1]}(x)$ and $w_2(x)\chi_{[0,1]}(x)$
defined on $\RR$. The finite integral domain can be considered as a consequence of the finite supports of
$w_1$ and $w_2$.  %\sotodo{added a sentence}
More generally, orthogonal polynomials on a wedge can be considered as a 
special case of orthogonal polynomials on two intersecting lines. 

\subsection{Orthogonal structure when $w_1 = w_2$} 

Let $w$ be a weight function defined on $\RR$. We assume $w_1= w_2 = w$ and define
\begin{equation} \label{eq:ipd-lines2}
    \la f,g\ra : = \int_\RR f(x,0) g(x,0) w(x) dx +  \int_\RR f(0,y) g(0,y) w(y) dy, 
\end{equation}
In this case, the study in \cite{OX} can be adopted with little change. 

\begin{thm} \label{thm:OP-lines}
Let $w$ be a weight function on $\RR$ and let $w_1(x): = x^2 w(x)$. Define 
\begin{align} \label{eq:OP-lines}
\begin{split}
  Y_{n,1} (x,y) & = p_n(w;x)+ p_n(w;y) - p_n(w;0), \quad n= 0,1,2,\ldots,\\
  Y_{n,2} (x,y) & =  x p_{n-1}(w_1; x) -y p_{n-1}(w_1; y), \quad n=1,2, \ldots. 
\end{split}
\end{align}
Then $\{Y_{n,1}, Y_{n,2}\}$ are two polynomials in $\CH_n(\varpi)$ and they are mutually orthogonal.  
Furthermore, 
\begin{equation*} %\label{eq:norm-lines}
  \la Y_{n,1}, Y_{n,1}\ra = 2 h_n(w_0) \quad\hbox{and}\quad  \la Y_{n,2}, Y_{n,2}\ra = 2 h_{n-1}(w_1).  
\end{equation*}
\end{thm}

The proof is identical to that of \cite[Theorem 2.2]{OX}. Furthermore, the Fourier orthogonal expansions
on the wedge is studied in \cite{OX}, which can be entirely adopted to the current setting. Let $L^2(\varpi,\Omega)$ 
be the $L^2$ space on $\Omega$ with norm $\|\cdot \|_{L^2(\varpi,\Omega)}$ defined by 
$$
  \|f\|_{L^2(\varpi,\Omega)}^2 =  \int_\RR |f(x,0)|^2 w(x) dx +  \int_\RR |f(0,y)|^2 w(y) dy.
$$

\begin{thm} \label{thm:Snf-intersection}
Let $f$ be a function defined on $\Omega$. Define 
$$
 f_e(x) := {f(x,0) + f(0,x) \over 2} \quad \hbox{and} \quad f_o(x) := \frac12\frac{f(x,0)-f(0,x)}{x}. 
$$
Then
\begin{align}
  S_n f(x,0) & \,= s_n(w; f_e, x) + x s_{n-1} (w_1; f_o, x),  \label{eq:Sn=1}  \\
  S_n f(0,y) & \,= s_n(w; f_e, y) - y s_{n-1} (w_1; f_o, y).  \label{eq:Sn=2} 
\end{align}
In particular, if $s_n(w;f_e,x) \to f_e(x)$ and $s_n(w_1;f_o,x) \to f_o(x)$ as $n \to \infty$, 
then $S_n f(x,y)$ converges to $f(x,y)$ as $n \to \infty$. Furthermore, if $f\in L^2(\varpi,\Omega)$, 
then 
$$
\|f - S_n(f) \|_{L^2(\varpi,\Omega)}^2 = 2 \left( \|s_{n-1} (w;f_e) - f_e\|_{L^2(w_0)}^2 
    + \|s_{n-1} (w_1;f_o) - f_o\|_{L^2(w_1)}^2 \right).
$$
\end{thm}

This theorem is equivalent to \cite[Theorem 2.4]{OX} and its corollary. 

As an example, we consider the case when $w$ is the Gaussian $w(x) = e^{-x^2}$. We shall need
the Hermite polynomials $H_n$ and the Laguerre polynomials $L_n^{\a}$.

\begin{prop}  %\sotodo{new example in response to item 10 of the report}
For $w(x) = e^{-x^2}$ the orthogonal polynomials of two variables in Theorem \ref{thm:OP-lines} are
given by
\begin{align*}% \label{eq:OP-lines}
\begin{split}
  Y_{n,1} (x,y) & = H_n(x)+ H_n(y) - H_n(0), \quad n= 0,1,2,\ldots,\\
  Y_{2n,2} (x,y) & =  x^2 L_{n-1}^{\f32}(x^2) - y^2 L_{n-1}^{\f32}(y^2), \quad n=1,2, \ldots,  \\
  Y_{2n+1,2} (x,y) & =  x L_n^{\f12}(x^2) - y L_n^{\f12}(y^2), \quad n= 0, 1, 2, \ldots. 
\end{split}
\end{align*}
\end{prop}

\begin{proof} 
The Hermite polynomials $H_n$ are orthogonal with respect to $w(x)$, so $Y_{n,1}$ is trivial. We 
need to find $p_n(w_1)$ for $w_1(x) = x^2 e^{-x^2}$. which we claim to be 
$$
p_{2n}(w_1; x) = x^2 L_{n-1}^{\f32} (x^2) \quad \hbox{and} \quad p_{2n+1}(w_1; x) = x L_n^{\f12} (x^2).  
$$
Indeed, the parity of these polynomials leads to the orthogonality of $p_{2n}(w_1)$ and $p_{2m+1}(w_1)$. 
Changing variable $x^2 \to y$ shows that the orthogonality of $p_{2n}(w_1)$ to even polynomials 
is equivalent to that of $L_{n-1}^{\f32}$ and the orthogonality of $p_{2n+1}(w_1)$ to odd polynomials 
is equivalent to that of $L_{n}^{\f12}$.
\end{proof}

\subsection{Orthogonal structure on wedges and on intersection lines}  %\sotodo{This is changed}
In this subsection we show that orthogonal structure on two intersecting lines can 
be derived from orthogonal structure on a wedge, which works even if $w_1 \ne w_2$. 

Let $w_1$ and $w_2$ be even functions on $\RR$, which we write as 
$$
w_1(t) = u_1(t^2) \quad \hbox{and} \quad w_2(t) = u_2(t^2), 
$$
so that the inner product \eqref{eq:ipd-lines} becomes 
\begin{equation} \label{eq:ipd-lines3}
   \la f, g\ra_{w_1,w_2} = \int_{\RR} f(x,0) g(x,0) u_1(x^2) dx + \int_{\RR} f(0,y) g(0,y) u_2(y^2) d y.
\end{equation}
Let $\CH_n(\varpi)$ be the space of orthogonal polynomials of degree $n$ with respect to this inner product. 
A basis in $\CH_n(\varpi)$ can be derived from orthogonal polynomials 
on the wedge with respect to the inner product
\begin{equation} \label{eq:ipd-wedge3}
 \la f, g\ra_{u_1,u_2}^{\rm wdg} = \int_{0}^\infty f(x,0) g(x,0) u_1(x) dx + \int_{0}^\infty f(0,y) g(0,y) u_2(y) d y. 
\end{equation}
We denote an orthogonal basis for $\CH_n^{\rm wdg}(u_1,u_2)$ by $Y_{n,1}^{u_1,u_2}$ and $Y_{n,2}^{u_1,u_2}$ to emphases
the dependence on the weight functions, as different weight functions will be used below.  

\begin{thm} \label{thm:OP-lines2}
Let $u_1$ and $u_2$ be weight functions defined on $\RR_+$. For $i  = 1, 2$, define $\phi u_i(t) = t^{-\f12} u_i(t)$ and 
$\psi u_i(t) = t^{\f12} u_i(t)$. For $n =1,2, \ldots,$ the polynomials 
\begin{equation} \label{eq:OP-lines3e}
   Y_{2 n,1}(x,y) = Y_{n,1}^{\phi u_1, \phi u_2} (x^2,y^2), \qquad  Y_{2 n,2} (x,y) = Y_{n,2}^{\phi u_1,\phi u_2} (x^2,y ^2)
\end{equation}
form an orthogonal basis for $\CH_{2n}(w_1,w_2)$ and the polynomials
\begin{align} \label{eq:OP-lines3o}
   Y_{2 n+1,1}(x,y) = (x+y) Y_{n,1}^{\psi u_1, \psi u_2} (x^2,y^2),
      \qquad  Y_{2 n+1,2}(x,y) = (x+y) Y_{n,2}^{\psi u_1, \psi u_2} (x^2,y^2).
\end{align}
form an orthogonal basis for $\CH_{2n+1}(w_1,w_2)$. 
\end{thm}

\begin{proof}
For $i,j =1,2$, changing variables $s = x^2$ and $t = y^2$ gives 
\begin{align*}
  \la Y_{2n,i}, Y_{2m,j} \ra_{w_1,w_2} = &\,  \int_0^\infty Y_{n,i}^{\phi u_1,\phi u_2} (s,0) 
     Y_{n,j}^{\phi u_1,\phi u_2} (s,0)  s^{-\f12}u_1(s)  ds  \\
     &  +   \int_0^\infty Y_{n,i}^{\phi u_1,\phi u_2} (0,t) 
     Y_{n,j}^{\phi u_1,\phi u_2} (0,t)  t^{-\f12}u_2(t)  dt \\
  = & \, \la Y_{n,i}^{\phi u_1, \phi u_2}, Y_{m,j}^{\phi u_1, \phi u_2} \ra_{\phi u_1,\phi u_2}^{\rm wdg},
\end{align*}
which verifies the statement for even polynomials. Since $Y_{2n+1,i}(x,0)$ and $Y_{2n+1,i}(0,y)$ are odd
polynomials in $x$ and in $y$, respectively, the parity of the polynomials shows that 
$\la Y_{2n+1,i}, Y_{2m,j} \ra_{w_1,w_2} =   0$. Furthermore, for $i = 1,2$, 
\begin{align*}
  \la Y_{2n+1,i}, Y_{2m+1, i} \ra_{w_1,w_2} = &\,  \int_\RR  x^2 Y_{n,i}^{\psi u_1, \psi u_2}(x^2,0) 
    Y_{m,i}^{\psi u_1, \psi u_2}(x^2,0) u_1(x^2) dx \\
      & +  \int_\RR  y^2 Y_{n,i}^{\psi u_1, \psi u_2}(0, y^2) Y_{m,i}^{\psi u_1, \psi u_2}(0, y^2) u_2(y^2) dx \\
     = & \, \la Y_{n,i}^{\psi u_1, \psi u_2}, Y_{m,j}^{\psi u_1, \psi u_2} \ra_{\psi u_1,\psi u_2}^{\rm wdg},
\end{align*}
where we have made the change of variables $s = x^2$ and $t = y^2$ again in the last step. This completes the 
proof. 
\end{proof}

Let $L^2(\varpi,\Omega)$ be the $L^2$ space on $\Omega$ with the norm $\|\cdot\|_{L^2(\varpi,\Omega)}$ 
defined by 
\begin{align*}
  \|f\|_{L^2(\varpi, \Omega)}^2: =\la f,f\ra 
      = \int_1^\infty  \left[ |f(x, \sqrt{x^2-1})|^2 + |f(x, - \sqrt{x^2-1})|^2 \right]w(x) dx.
\end{align*}
Furthermore, let $L^2(u_1,u_2)$ be the $L^2$ space on the wedge $\{(x, 0): x \ge 0\} \cup \{(0, y): y \ge 0\}$ with
the norm $\|\cdot\|_{L^2(u_1,u_2)}$ defined by
\begin{equation*}
 \| f \|_{u_1,u_2}^{\rm wdg}:= \int_{0}^\infty |f(x,0)|^2 u_1(x) dx + \int_{0}^\infty |f(0,y)|^2 u_2(y) d y. 
\end{equation*}
Let $s_{n}^{u_1, u_2}f$ denote the $n$-th partial sum of the Fourier orthogonal expansions with respect to 
\eqref{eq:ipd-wedge3} on the wedge. 
 
\begin{thm} \label{thm:Snf-Iines}
Let $f$ be a function defined on $\Omega$. Define 
$$
 f_e(x,y) := \frac{f(x,y) + f(-x,-y)}{2} \quad \hbox{and} \quad f_o(x,y) := 
     \frac{f(x,y) - f(-x,-y) }{2 (x+y)}.
$$
Furthermore, let $\Phi(x,y) := (\sqrt{x}, \sqrt{y})$. Then the partial sum operator for the Fourier orthogonal series with 
respect to \eqref{eq:ipd-lines3} satisfies
\begin{align} \label{eq:Snf-lines}
  S_n f(x,y) \,= s_{\lfloor \f{n}{2}\rfloor}^{\phi u_1, \phi u_2}(f_e \circ \Phi) (x^2, y^2) 
        + (x+y) s_{\lfloor \f{n-1}{2}\rfloor}^{\psi u_1, \psi u_2} (f_o \circ \Phi) (x^2, y^2). 
\end{align}
Furthermore, if $f\in L^2(w_1,w_2,\Omega)$, then 
\begin{align*} %\label{eq:SnL2-hyper2}
\|f - S_n(f) \|_{L^2(w_1,w_2,\Omega)}^2 = \|s_{\lfloor \f{n}{2}\rfloor}^{\phi u_1, \phi u_2} f_e - f_e\|_{L^2(\phi u_1, \phi u_2)}^2  
      + \|s_{\lfloor \f{n-1}{2}\rfloor}^{\psi u_1, \psi u_2} f_o - f_o \|_{L^2(\psi u1,\psi u_2)}^2.  
\end{align*}
\end{thm}

\begin{proof}
Following the proof of the previous theorem, we see that 
\begin{align*}
  \la f, Y_{2n,j} \ra_{w_1,w_2} = &\,  \int_0^\infty  [f(x,0) + f(-x,0)] Y_{n,j}^{\phi u_1,\phi u_2} (x^2,0) u_1(x^2)  dx  \\
       &  +   \int_0^\infty [f(0,y)+f(0,-y)] Y_{n,j}^{\phi u_1,\phi u_2} (0,y^2) u_2(y^2)  d y \\
    = &\,  \int_0^\infty  (f_e \circ \Phi) (s,0) Y_{n,j}^{\phi u_1,\phi u_2} (s,0) s^{-\f12} u_1(s)  ds  \\
       &  +  \int_0^\infty (f_e \circ \Phi) (0,t) Y_{n,j}^{\phi u_1,\phi u_2} (0,t) t^{-\f12}u_2(t)  d t \\
   =  & \, \la f_e \circ \Phi, Y_{n,i}^{\phi u_1, \phi u_2} \ra_{\phi u_1,\phi u_2}^{\rm wdg}.
\end{align*}
Similarly, we can verify that 
\begin{align*}
  \la f, Y_{2n+1, i} \ra_{w_1,w_2} = \la f_o \circ \Phi, Y_{n,i}^{\psi u_1, \psi u_2}\ra_{\psi u_1,\psi u_2}^{\rm wdg},
\end{align*}
since $(f_\circ \Phi)(x^2,0) = (f(x,0)- f(-x,0))/ (2x)$ and  $(f_\circ \Phi)(0, y^2) = (f(0,y)- f(0,-y))/ (2y)$. In particular,
together with the identities in the proof of the previous theorem, we obtain
$$
  \wh f_{2n,i} =  \wh{f_e \circ \Phi}^{\phi u_1,\phi u_2}_{n, i} \quad \hbox{and} \quad 
    \wh f_{2n+1,i} = \wh{f_o \circ \Phi}^{\psi u_1,\psi u_2}_{n, i},
$$
where the righthand sides are the Fourier coefficients with respect to the basis $Y_{n,i}^{\phi u_1,\phi u_2}$ and 
$Y_{n,i}^{\psi u_1,\psi u_2}$ on the wedge, respectively. The relation on the partial sums then follows from 
\eqref{eq:OP-lines3e} and \eqref{eq:OP-lines3o}.
\end{proof}

\subsection{Orthogonal structure for Jacobi weight functions}
In this example, we consider the example when both $w_1$ and $w_2$ are 
Jacobi weight functions with support sets on $[0,1]$. In other words, we consider
$$
  w_1(t) = |t|^{2\g} (1-t^2)^\a \chi_{[0,1]}(t), \quad   w_2(t) = |t|^{2\g} (1-t^2)^\b \chi_{[0,1]}(t).
$$
For $\a,\b, \g > -1$, we renormalize the the inner product \eqref{eq:ipd-wedge3} on the wedge as 
\begin{align}
  \la f,g\ra_{\a,\b,\g}:=  &\,  c_{\a,\g} \int_0^1 f(x,0) g(x,0) |x|^{2\g} (1-x^2)^\a dx  \\
       & \quad  +   c_{\b,\g} \int_0^1 f(0,y) g(0,y)  |y|^{2\g} (1-y^2)^\b dy,  \notag
\end{align}
where 
$$
  c_{\a,\g} := \Big(\int_{-1}^1 t^{2\g} (1-t^2)^\a dx \Big)^{-1} = \frac{\Gamma(\g+\a+\f32)}{\Gamma(\g+\f12)\Gamma(\a+1)}. 
$$
Orthogonal polynomials with respect to this inner product can be derived from \cite[Theorem 3.2]{OX} by making 
a change of variables $(x,y) \mapsto (1-x,1-y)$. The result is as follows: 

\begin{thm} \label{thm:OP-ipd-abg}
Let $Y_{0,1}(x,y) =1$, $Y_{1,1}(x,y) = x$ and $Y_{1,2}(x,y) = y$. For $n =1,2,\ldots$, let 
\begin{align*}
 Y_{2n,1}(x,y) = & \,P_n^{(\g-\f12,\a)}(x^2) + P_n^{(\g-\f12,\b)}(y^2) -  \binom{n+\g-\f12}{n},\\
 Y_{2n,2}(x,y) = & \, \frac{(\g+\a+\f32)_{n}}{(\a+1)_{n-1}} xP_{n-1}^{(\g+\f32,\a)}(x^2) 
                    - \frac{(\g+\b +\f32)_{n}}{(\b+1)_{n-1}} y P_{n-1}^{(\g+\f32,\b)}(y^2), 
\end{align*}
and 
\begin{align*}
 Y_{2n+1,1}(x,y) = & \,(x+y) \left[ P_n^{(\g+\f12,\a)}(x^2) + P_n^{(\g+\f12,\b)}(y^2) -  \binom{n+\g+\f12}{n} \right], \\
 Y_{2n+1,2}(x,y) = & \, (x+y) \left[ \frac{(\g+\a+\f32)_{n+1}}{(\a+1)_{n-1}} xP_{n-1}^{(\g+\f52,\a)}(x^2) 
                    - \frac{(\g+\b +\f32)_{n+1}}{(\b+1)_{n-1}} y P_{n-1}^{(\g+\f52,\b)}(y^2) \right]. 
\end{align*}
Then $Y_{n,1}, Y_{n,2}$ form a basis in $\CH_n(w_1,w_2)$ and 
\begin{align*}% \label{eq:ipdPQ}
  \la Y_{2n,1}, Y_{2n,2} \ra_{\a,\b,\g} &\, = \frac{(\b-\a) (\g+\f12)_{n+1}}{(2n+\g+\a+\f12)(2n+\g+\b+\f12) (n-1)!}, \\
  \la Y_{2n+1,1}, Y_{2n+1,2} \ra_{\a,\b,\g} &\, 
        = \frac{(\b-\a) (\g+\f12)_{n+2}}{(2n+\g+\a+\f32)(2n+\g+\b+\f32) (n-1)!}.
\end{align*}
\end{thm}

\begin{proof}
With our notation $w_1(t) = u_1(t^2)$ and $w_2(t) = u_2(t^2)$, we have 
$$
  u_1(t) = t^{\g} (1-t)^\a, \quad   u_2(t) = t^{\g} (1-t)^\b, \quad t \in [0,1]. 
$$
With these weight functions, we renormalize the inner product \eqref{eq:ipd-wedge3} as 
\begin{align*}
  \la f,g\ra_{\rm wdg}:=  c_{\a,\g}' \int_0^1 f(x,0) g(x,0) x^{\g} (1-x)^\a dx  
           +   c_{\b,\g}' \int_0^1 f(0,y) g(0,y) y^{\g} (1-y)^\b dy,  
\end{align*}
where $c_{\a,\g}' = c_{\a,\g+\f12}$ is the normalization constant of $t^{\g} (1-t)^\a$. Orthogonal 
polynomials for this inner product \eqref{eq:ipd-wedge3} on the wedge 
$\{(x,0): 0\le x \le 1\} \cup \{(0,y): 0\le y \le 1\}$ are given in \cite[Theorem 3.2]{OX} up to a simple 
change of variables $(x,y) \mapsto (1-x,1-y)$, since the wedge considered in \cite{OX} is 
$\{(x,1): 0\le x \le 1\} \cup \{(1,y): 0\le y \le 1\}$. 

Since $\phi u_1(t) = t^{\g-\f12} (1-t)^\a$ and $\phi u_2(t) = t^{\g-\f12} (1-t)^\b$, we then obtain orthogonal 
polynomials $Y_{n,i}^{\phi u_1, \phi u_2}$ from the results in \cite{OX} by considering $\g \mapsto \g - \f12$,
which gives the formula $Y_{2n,i}$ by Theorem \ref{thm:OP-lines2}, where we have used the fact that 
$c_{\a,\g} = c'_{\a,\g-\f12}$ so that $\la 1,1\ra_{\a,\b,\g}  = \la 1,1\ra_{\rm wdg}$ holds. Furthermore, since 
$\psi u_1(t) = t^{\g+\f12} (1-t)^\a$ and $\psi u_2(t) = t^{\g+\f12} (1-t)^\b$, we obtain orthogonal 
polynomials $Y_{n,i}^{\psi u_1, \psi u_2}$ from the results in \cite{OX} by considering $\g \mapsto \g + \f12$. 
In this case, however, a shift in the index is necessary because of the normalization of the constants. Indeed,
in order to reduce the orthogonality of $Y_{2n+1,2}$ with respect to $\la \cdot,\cdot\ra_{\a,\b,\g}$ from the
orthogonality with respect to $\la \cdot,\cdot \ra_{\rm wdg}$, we have incorporated the constant
$$
   \frac{c_{\a,\g}} {c_{\a,\g+1} } =  \frac{c_{\a,\g-\f12}'}{c_{\a,\g+\f12}' } = \frac{\g+\f12}{\g+\a+\f32}
$$
in the constant in front of $xP_{n-1}^{(\g+\f52,\a)}(x^2) $ in its expression, and the similar change is made on
the second term of its expression. 
\end{proof}

The convergence of the Fourier orthogonal expansions with respect to the Jacobi weight functions in
$L^2(u_1,u_2)$ on the wedge is established in \cite[Theorem 3.4]{OX}, from which we can derive the convergence 
of the Fourier orthogonal expansions for the Jacobi weight functions in $L^2(w_1,w_2, \Omega)$ from that
Theorem \ref{thm:Snf-Iines}. The derivation is straightforward but the statement is somewhat complicated. 
We leave it to interested readers. 

As another example, we can consider the case when $w_1$ and $w_2$ are Laguerre weights, 
 %\sotodo{added in response to item 10 of the report}
$$
   w_1(x) = x^\a e^{-x} \chi_{[0,\infty)}(x) \quad \hbox{and}\quad w_2(x) = x^\b e^{-x} \chi_{[0,\infty)}(x), \quad \a,\b> -1.
$$
This can be developed similarly as in the Jacobi case. 

\section{Two parallel lines}
\setcounter{equation}{0}

Two parallel lines can be transformed to the lines $y =1$ and $y=-1$ by an affine transform. We
then consider 
$$
  \Omega = \{(x, y): y = 1 \, \, \hbox{and} \, \, y = -1\}
$$ 
Let $w$ be a weight functions defined on the real line. Define a bilinear form
\begin{equation} \label{eq:ipd-parallel}
    \la f,g\ra_{w} = \int_\RR \left[ f(x,-1) g(x,-1) +  f(x,1) g(x,1)\right] w(x) dx, 
\end{equation}
which is an inner product for $\RR[x,y]/\la y^2-1 \ra$. 

\begin{thm}\label{thm:ParellelLines}
Let $w$ be a weight function on $\RR$ and let $w_1(x): = x^2 w(x)$. Define 
\begin{align} \label{eq:OP-parallel}
\begin{split}
  Y_{n,1} (x,y) & = p_n(w;x), \quad n= 0,1,2,\ldots,\\
  Y_{n,2} (x,y) & =  y p_{n-1}(w_1; x), \quad n=1,2, \ldots. 
\end{split}
\end{align}
Then $\{Y_{n,1}, Y_{n,2}\}$ are two polynomials in $\CH_n(\varpi)$ and they are mutually orthogonal.  
\end{thm}

\begin{proof}
It is evident that $\la Y_{n,1}, Y_{m,2} \ra =0$ since there is only one $y$. Furthermore, 
$\la Y_{n,1}, Y_{m,1} \ra =  2 h_n(w_0)\delta_{n,m}$ and $\la Y_{n,2}, Y_{m,2} \ra =  2 h_{n-1}(w_0)\delta_{n,m}$
by the orthogonality of $p_n(w_0)$. 
\end{proof}

For the Hermite and Laguerre weights, we note that the resulting orthogonal polynomials are 
eigenfunctions of simple partial differential operators.

\begin{prop}
For the Hermite weight function $w(x) = e^{-x^2}$, the orthogonal polynomials in \eqref{eq:OP-parallel}
are the eigenfunctions of a second order differential operator; more precisely, the polynomials
$H_n(x)$ and $y H_{n-1}(x)$ satisfy the equation
\begin{equation} \label{eq:PDE-Hermite2}
  \CL u : =   \partial_{x x}u - 2 x \partial_x u - 2 y \partial_y u = - 2 n u. 
\end{equation}
\end{prop}

\begin{proof}
The Hermite polynomial $v = H_n$ satisfies the equation
$$
       v ''  -2 x v' = - 2 n v.
$$
The verification of \eqref{eq:PDE-Hermite2} for $u(x,y) = H_n(x)$ follows immediately since $\partial_y u =0$
and also for $u(x,y) = y H_{n-1}(x)$ since $y \partial_y u = u$. 
\end{proof}

\begin{prop}
For the Laguerre weight function $w(x) = e^{-x}$, the orthogonal polynomials in \eqref{eq:OP-parallel}
are the eigenfunctions of a second order differential operator; more precisely, the polynomials
$L_n^\a (x)$ and $y L_{n-1}^\a (x)$ satisfy the equation
\begin{equation} \label{eq:PDE-Laguerre2}
  \CL u : =  \partial_{xx}u+ (\a+1 - x) \partial_x u - y \partial_y u = -n u. 
\end{equation}
\end{prop}

\begin{proof}
The verification follows as the previous proposition, using the fact that the Laguerre polynomial $v = L_n^\alpha$ 
satisfies the equation $v '' + (\a +1 - y) u' = -n y.$
\end{proof}

\section{Interpolation}

Consider the problem of interpolating a function $f(x,y)$ using any of the families of orthogonal polynomials constructed 
above, which we denote as $Y_{n,1}(x,y)$ and $Y_{n,2}(x,y)$, and without loss of generality we for now assume our geometry
is invariant under reflection across the $x$ axis, so that if $(x,y) \in \Omega$ then so is $(x,-y)$ (e.g., the one-branch Hyperbola or the circle). 
 Given a set of $2n$ interpolation points  
$\{(x_j,\pm y_j)\}_{j=1}^n$, consider the problem of finding coefficients $f_{k,1}^n$ for $k = 0, \ldots, n-1$ and $f_{k,2}^n$ 
for $k = 1, \ldots, n$, so that the polynomial\footnote{The choice of having the term $Y_{n,2}$  as opposed to $Y_{n,1}$ is in some sense arbitrary, but may be required depending on the choice of grids. We have made this choice for concreteness in the discussion. Further, it is also possible to have an odd number of points, in which case we would need neither extra term.}
\begin{equation}\label{eq:8.1}
f_n(x,y) = f_{0,1}^n  Y_{0,1}(x,y) + \sum_{k = 1}^{n-1} \left[Y_{k,1}(x,y) f_{k,1}^n+ Y_{k,2}(x,y) f_{k,2}^n \right] + f_{n,2}^n Y_{n,2}(x,y)
\end{equation}
interpolates $f(x,y)$ at $x_j, \pm y_j$:
$$
f_n(x_j, \pm y_j) = f(x_j, \pm y_j).
$$

A straightforward method to determine the coefficients in \eqref{eq:8.1} is to solve a linear system: 
$$
	V \begin{pmatrix} f_{0,1}^n \\ f_{1,2}^n \\ f_{1,1}^n \\ \vdots \\ f_{n-1,1}^n \\ f_{n,2}^n  \end{pmatrix} = \begin{pmatrix} f(x_1,y_1) \\ \vdots\\ f(x_n,y_n) \\ f(x_1,-y_1) \\ \vdots\\ f(x_n,-y_n) \end{pmatrix} 
$$
where $V$ is the $2n \times 2n $ interpolation matrix
$$
V = \begin{pmatrix}
Y_{0,1}(x_1,y_1) & \cdots & Y_{n-1,1}(x_1,y_1) & Y_{n,2}(x_1,y_1) \\
\vdots & \ddots & \vdots & \vdots \\
Y_{0,1}(x_n,y_n) & \cdots & Y_{n-1,1}(x_n,y_n) & Y_{n,2}(x_n,y_n) \\
Y_{0,1}(x_1,-y_1) & \cdots & Y_{n-1,1}(x_1,-y_1) & Y_{n,2}(x_1,-y_1) \\
\vdots  & \ddots & \vdots & \vdots \\
Y_{0,1}(x_n,-y_n) & \cdots & Y_{n-1,1}(x_n,-y_n) & Y_{n,2}(x_n,-y_n) 
\end{pmatrix}.
$$
However, solving this system requires $O(n^3)$ operation for $n$ interpolation points,  we have no guarantees the system is invertible, and it is not immediately amenable to analysis.

In this section we give two alternative constructions of the interpolation polynomial. 
The first construction  is based on interpolation polynomials of one variable, which works for general interpolation
grids and allows us to relate one-dimensional interpolation results to this higher dimensional setting.
The second construction is based on Gaussian
quadrature rules which allows efficient construction of interpolants at specific grids generated from Gaussian quadrature. 

\subsection{Interpolation on quadratic curves}  %\sotodo{New subsection}
In this subsection we show that interpolation polynomials on
a quadratic curve can be expressed via interpolation polynomials of one variable, in a way that lets us bootstrap one variable results to this setting. 

First we recall classical results on the real line. Let $w$ be a weight function on $\RR$. Let $p_k(w)$ be the $k$-th
orthogonal polynomial with respect to $w$.  Let $\{t_j: 1 \le j \le n\}$ be the set of zeros of $p_n(w)$. The 
unique Lagrange interpolation polynomial of degree $n-1$ based on these zeros is given by 
$$
    L_n (w; f, t) = \sum_{k=1}^n f(t_j) \ell_{j,n}(t), \qquad \ell_{j,n} (t) = \frac{p_n(w;t)}{p_n'(w;t_j) (t-t_j)}. 
$$
Let $K_n(w;\cdot,\cdot)$ denote the reproducing kernel
$$
   K_n(w;x,y) = \sum_{k=0}^n \frac{p_k(w;x) p_k(w;y)}{h_k(w)}, \qquad  h_k(w) = \int_\RR [p_k(w)]^2 w(x) dx
$$
of the space of polynomials of degree at most $n$. By the Christoffel--Darboux formula, the fundamental 
interpolation polynomial $\ell_{j,n}$ is also equal to 
\begin{equation*}
    \ell_{j,n}(t) =  \frac{K_{n-1}(w;t_j,x)}{K_{n-1}(w;t_j,t_j)} = \l_{j,n}  K_{n-1}(w;t_j,x),
\end{equation*}
where $\l_{j,n} = [K_{n-1}(w;t_j,x)]^{-1}$ are the weights of the Gaussian quadrature rule. By the reproducing 
property of $K_{n-1}$ and the Christoffel--Darboux formula, it is easy to see that 
\begin{equation}\label{eq:interp2}
     \int_{-1}^1 \ell_{j,n}(t)  \ell_{k,n}(t) w(t) dt = \lambda_{j,n} \delta_{j,k}. 
\end{equation}

We now consider interpolation polynomial $\CL_{2n}f$ on the circle, as defined in Propositions 
\ref{prop:interpolation} and \ref{prop:circle-interp1}. As we have shown, $\CL_{2n} f(x,y)$ is the 
unique polynomial in the space 
$$
  \CP_n:= \mathrm{span} \{Y_{k,1}: 0 \le k \le n-1\} \cup \{Y_{k,2}: 1 \le k \le n\} = \Pi_{n-1}^{(x)} \oplus y \Pi_{n-1}^{(x)},
$$
that interpolates $f$ at $(x_j,y_j)$ and $(x_j, - y_j)$, $1\le j \le n$, where $x_j$, $1 \le j \le n$
are zeros of the orthogonal polynomial $p_n(w_{-\f12})$ on $[-1,1]$,  $y_j = \sqrt{1-x_j^2}$
and the superscript $(x)$ in $\Pi_{n-1}^{(x)}$ indicates that it is a space of polynomials in $x$ variable. 
It turns out that $\CL_{2n} f$ admits a simple expression in terms of interpolation polynomials of one variable.   
 
\begin{prop}[Circle]\label{prop:circle-interp2}
Let $w_{-\f12}$ be defined on $[-1,1]$ as in Theorem~\ref{thm:circlebasis}. Define $f_e$ and $f_o$ as 
in Theorem \ref{thm:Snf-circle}. Then the interpolation polynomial $\CL_{2n} f$ can be written as
\begin{equation}\label{eq:circle2}
   \CL_n f(x,y) = L_n\left(w_{-\f12}; f_e, x\right) + y L_n \left(w_{-\f12}; f_o, x\right). 
\end{equation}
\end{prop}

\begin{proof}
By its definition, $\ell_j(x) \in \Pi_{n-1}{(x)}$, so that $\ell_j \in \CP_n$ and $y \ell_j \in \CP_n$. Hence, 
$\CL_n$ is an element of $\CP_n$. Furthermore,
$$
  \CL_n f(x_j,y_j) =  f_e (x_j) \ell_j(x_j) + y_j  f_o (x_j) \ell_j(x_j) = f(x_j, y_j) 
$$
and 
$$
  \CL_n f(x_j,- y_j) =  f_e(x_j) \ell_j(x_j) - y_j f_o(x_j) \ell_j(x_j) = f(x_j, -y_j) 
$$
for $1 \le j \le n$. Hence, the interpolation conditions are satisfied. 
\end{proof}

For the classical trigonometric series, interpolation in the form \eqref{eq:circle2} was discussed in 
\cite[Vol. II, Section X.3]{Z}.

The expression of $\CL_{2n} f$ in Proposition \ref{prop:circle-interp2} allows us to utilize results on interpolation of
one variable. As an example, we can deduce a mean convergence theorem for $\CL_n$. For a continuous function
$g: [-1,1]\to \RR$, let the error of best approximation by polynomials of degree at most $m$ be denoted by 
$$
E_m(g)_\infty := \inf_{p_m \in \Pi_m} \|g- p_m\|_\infty. 
$$

\begin{thm}
Let $f(x,y)$ be a function defined on the unit circle. Then 
$$
   \|\CL_n f\|_{L^2(\varpi, \Omega)} \le \max_{1\le j \le n} |f_e(x_j)| + \max_{1\le j \le n} |f_o(x_j)|. 
$$
In particular, if both $f_e$ and $f_o$ are continuous, then $\|\CL_n f - f\|_{L^2(w)} \to 0$. Moreover, 
$$
   \|\CL_n f - f\|_{L^2(\varpi, \Omega)} \le 2 E_{n-1}(f_e)_\infty + 2 E_{n-1}(f_o)_\infty.
$$
\end{thm}

\begin{proof}
We write $L_n f (x) = L_n\left(w_{-\f12}; f, x\right)$ in this proof. By the orthogonality of $\Pi_{n-1}^{(x)}$ 
and $y\Pi_{n-1}^{(x)}$, we obtain 
\begin{align*}
 \|\CL_n f\|_{L^2(\varpi, \Omega)} = & \, c_w \int_{-\pi}^\pi |\CL_n f(\cos \t,\sin\t)|^2 w(\cos \t) d\t \\
 = &\, 2 c_w \int_{-1}^1 |L_n f_e(t)|^2 w_{-\f12}(t) dt + 2c_w \int_{-1}^1 |L_n f_o(t)|^2 (1-t^2) w_{-\f12}(t) dt.
\end{align*}
A classical result derived using \eqref{eq:interp2} shows that 
$$
2  c_w \int_{-1}^1 \left |L_n\left(w_{-\f12}; f_e, x\right)) \right|^2 w_{-\f12}(t) dt 
= \sum_{j=1}^n \l_{j,n} |f_e(x_j|^2 \le \max_{1\le j \le n} |f_e(x_j)|^2,
$$
and
$$
2  c_w \int_{-1}^1 |L_n f_o(t)|^2  (1-t^2) w_{-\f12}(t) dt \le 2  c_w \int_{-1}^1 |L_n f_o(t)|^2 w_{-\f12}(t) dt
       \le  \max_{1\le j \le n} |f_o(x_j)|^2.
$$

To prove the convergence, we use the fact that $\CL_nf = f$ if $f \in \CP_n$. Hence, for $P(x,y) = q_1(x)+ y q_2(x)$ 
with $q_1, q_2 \in \Pi_{n-1}$, we obtain from $f(x,y) = f_e(x) + y f_o(x)$ that 
$$
  \|f-P\|_{L^2(\varpi, \Omega)} \le \|f_e - q_1\|_{L^2(w)} + \|y (f_o - q_2) \|_{L^2(w)} \le \|f_e - q_1\|_\infty + \|f_o - q_2\|_\infty, 
$$
so that 
\begin{align*}
  \|\CL_n f - f\|_{L^2(\varpi, \Omega)} & = \|\CL(f- P_n) - (f-P_n) \|_{L^2(w)} \\
     & \le \|\CL(f- P_n)\|_{L^2(w)} + \|f-P_n\|_{L^2(w)} \\
    & \le 2 \|f_e - q_1\|_\infty + 2 \|f_o - q_2\|_\infty. 
\end{align*}
Taking infimum over $q_1$ and $q_2$, respectively, we conclude that 
$$
   \|\CL_n f - f\|_{L^2(\varpi, \Omega)} \le 2 E_{n-1}(f_e)_\infty + 2 E_{n-1}(f_o)_\infty,
$$
which clearly goes to zero as $n \to \infty$. 
\end{proof}

A similar construction readily translates to each of the other quadratic curves. The cases of intersection
lines and parallel lines are more or less trivial, so we state the results only for non-degenerate quadratic
curves. Since the proofs are essentially the same as the circle case, we shall omit them. 

Let $\Omega = \{(x,y): y = x^2 , \  x \in \RR\}$ be the parabola. Let $\CL_n f$ be the interpolation polynomial 
defined in Proposition \ref{prop:interpolation} and Corollary \ref{prop:parabola-interp1}. It is the unique 
polynomial in the space 
$$
  \CP_n:= \mathrm{span} \{Y_{k,1}: 0 \le k \le n-1\} \cup \{Y_{k,2}: 1 \le k \le n\} = \Pi_{n-1}^{(y)}\cup x \Pi_{n-1}^{(y)},
$$
see Theorem~\ref{thm:parabolabasis}, that interpolates $f$ at $(x_j,y_j)$ and $(-x_j, y_j)$, $1\le j \le n$, 
where $y_j$, $1 \le j \le n$ are zeros of $p_n(w_{-\f12})$ on $\RR_+$ and $x_j = \sqrt{y_j}$. 
 
\begin{prop}[Parabola]
Let $f$ be defined on $\Omega$. Define $f_e(y)$ and $f_o(y)$ as in Theorem \ref{thm:Snf-parabola}. 
Then the interpolation polynomial $\CL_{2n} f$ can be written as
\begin{equation}\label{eq:parabolic-inter2}
   \CL_n f(x,y) = L_n \left(w_{-\f12};  f_e, y\right) + x L_n \left(w_{-\f12};  f_o, y\right). 
\end{equation}
\end{prop}
 
\begin{thm}
Let $f$ be a function defined on the parabola $\Omega$. Then 
$$
   \|\CL_n f\|_{L^2(\varpi, \Omega)} \le   \max_{1\le j \le n} |f_e(y_j)| + M \max_{1\le j \le n} |f_o(y_j)|, 
$$
if $w$ is supported at $[0, M]$ for some $M > 0$.  In particular, if $f_e$ and $f_o$ are continuous, 
then $\|\CL_n f - f\|_{L^2(\varpi, \Omega)} \to 0$. 
\end{thm}

Let $\Omega = \{(x,y):  x^2-y^2=1,\  y\in \RR\}$ be the hyperbola of two branches. 
Let $\CL_{2n} f$ be the interpolation polynomial defined in Proposition \ref{prop:interpolation} and
Corollary \ref{cor:twocutinterp}. It is the unique polynomial in the space 
$$
  \CP_n:= \mathrm{span} \{Y_{k,1}: 0 \le k \le n-1\} \cup \{Y_{k,2}: 1 \le k \le n\} = \Pi_{n-1}^{(y)}\cup x \Pi_{n-1}^{(y)},
$$
see Theorem~\ref{thm:TwoBranch}, that interpolates $f$ at $(x_j,y_j)$ and $(x_j, - y_j)$, $1\le j \le n$,
where $y_j$, $1 \le j \le n$ are zeros of $p_n(w)$ and $x_j = \sqrt{y_j^2+1}$. 

\begin{prop}[Hyperbola, two branches]
Let $f$ be defined on the hyperbola $\Omega$. Define $f_e(y)$ and $f_o(y)$ as in 
Theorem \ref{thm:Snf-hyper}. Then the interpolation polynomial $\CL_{2n} f$ can be written as
\begin{equation}\label{eq:hyper-inter2}
   \CL_n f(x,y) = L_n \left(w;  f_e, y\right) + x L_n \left(w;  f_o, y\right). 
\end{equation}
\end{prop}
  
\begin{thm}
Let $f$ be defined on the hyperbola $\Omega$ of two branches. Then 
$$
   \|\CL_n f\|_{L^2(\varpi, \Omega)} \le  \max_{1\le j \le n} |f_e(y_j)| + M^2 \max_{1\le j \le n} |f_o(y_j)|, 
$$
if $w$ is supported on $[-M, M]$ for some $M > 0$.  In particular, if both $f_e$ and $f_o$ are continuous, 
then $\|\CL_n f - f\|_{L^2(\varpi, \Omega)} \to 0$. 
\end{thm}

Let $\Omega = \{(x,y): x^2-y^2=1, \, x \ge 1\}$ be the hyperbola of one branch. 
Let $\CL_{2n} f$ be the interpolation polynomial defined in Proposition \ref{prop:interpolation} and
Corollary \ref{prop:hypbola-interp1}. It is the unique polynomial in the space 
$$
  \CP_n:= \mathrm{span} \{Y_{k,1}: 0 \le k \le n-1\} \cup \{Y_{k,2}: 1 \le k \le n\} = \Pi_{n-1}^{(x)}\cup y \Pi_{n-1}^{(x)},
$$
see Theorem~\ref{thm:OneBranch}, that interpolates $f$ at $(x_j,y_j)$ and $(x_j, - y_j)$, 
$1\le j \le n$, where $x_j$, $1 \le j \le n$ are zeros of $p_n(w_0)$ and $y_j = \sqrt{x_j^2+1}$. 

\begin{prop}[Hyperbola, one branch]
Let $f$ be defined on the hyperbola $\Omega$. Define $f_e(y)$ and $f_o(y)$ as in 
Theorem \ref{thm:OneBranch2}. Then the interpolation polynomial $\CL_{2n} f$ can be written as
\begin{equation}\label{eq:onebranch-inter2}
   \CL_n f(x,y) = L_n \left(w_0;  f_e, x\right) + y L_n \left(w;  f_o, x\right). 
\end{equation}
\end{prop}
 
\begin{thm}
Let $f$ be defined on the hyperbola $\Omega$ of one branch. Then 
$$
   \|\CL_n f\|_{L^2(\varpi, \Omega)} \le  \max_{1\le j \le n} |f_e(x_j| + M^2 \max_{1\le j \le n} |f_o(x_j|, 
$$
if $w_0$ is supported on $[1, M]$ for some $M > 0$. In particular, if $f_e$ and $f_o$ are continuous, 
then $\|\CL_n f - f\|_{L^2(\varpi, \Omega)} \to 0$. 
\end{thm}

\subsection{Interpolation via quadrature}

We will now determine these coefficients via a quadrature rule that preserves orthogonality. This construction enables 
the calculation of interpolation coefficients in $O(n^2)$ complexity. We begin with a general proposition that relates
quadrature rules to interpolation. Below, 
$\phi_n$ will be a carefully chosen ordering of $Y_{n,1}$ and $Y_{n,2}$, chosen to ensure orthogonality is preserved
and $M = 2n$.

\begin{prop}\label{prop:interpolation}
Suppose we have a discrete inner product for a basis $\{\phi_j \}_{j=0}^{M-1}$ of the form
$$
\la f, g \ra_M = \sum_{j=1}^M w_j f(x_j, y_j) g(x_j, y_j)
$$
satisfying
$
\la \phi_m, \phi_n  \ra_M = 0 
$
for $m \neq n$ and 
$
\la \phi_n, \phi_n \ra_M \neq 0.
$
Then the function
$$
\CL_M f(x,y) = \sum_{n=0}^{M-1} f_n^M \phi_n(x,y)
$$
interpolates $f(x,y)$ at $(x_j,y_j)$, where 
$
f_n^M := {\la \phi_n,f \ra_M\over \la \phi_n, \phi_n \ra_M}.
$
\end{prop}
\begin{proof}
Define the evaluation matrix
$$
E = \begin{pmatrix} 
	\phi_0(x_1,y_1) & \cdots & \phi_{M-1}(x_1,y_1) \\
	\vdots & \ddots & \vdots \\
	\phi_0(x_M,y_M) & \cdots & \phi_{M-1}(x_M,y_M) \\	
\end{pmatrix}
$$
The definition of $f_n^M$ can be written in matrix form
$$
P \begin{pmatrix} f(x_1) \\ \vdots \\ f(x_M) \end{pmatrix} = \begin{pmatrix} f_0^M \\ \vdots \\ f_{M-1}^M \end{pmatrix}
$$
where $P = N^{-1} E^\top W$ for 
$$
N = \begin{pmatrix} (\sum_{j=1}^M w_j \phi_0(x_j)^2) \\
		& \ddots \\
		&& (\sum_{j=1}^M w_j \phi_{M-1}(x_j)^2)
		 \end{pmatrix}, \quad 
		 W = \begin{pmatrix} w_1 \\ & \ddots \\&& w_M\end{pmatrix}.
$$
The hypotheses show that $P E = I$, i.e., $E P = I$ and therefore $P$ gives the coefficients of the interpolation. 
\end{proof}

Thus our aim is to construct a quadrature rule that preserves the orthogonality properties. We do so for each of the cases:

\begin{prop}[Circle] \label{prop:circle-interp1}
Let $\{t_j\}_{j=1}^n$ and $\{w_j\}_{j=1}^n$ be the $n$-point Gaussian quadrature nodes and weights with respect to $w_{-\f12}$ defined on $[-1,1]$, as in Theorem~\ref{thm:circlebasis}. Then the $2n$ polynomials $\{Y_{j,1}\}_{j=0}^{n-1}, \{Y_{j,2}\}_{j=1}^n$ are orthogonal with respect to the discrete inner product
$$
\la f,g\ra_n := \sum_{j=1}^n w_j \left[f(x_j, y_j)  g(x_j, y_j) +  f(x_j, -y_j)  g(x_j, -y_j)\right],
$$
where $x_j = t_j$ and  $y_j =  y_j = \sqrt{1-t_j^2}$. 
\end{prop}
\begin{proof}
For any $k$ and $\ell$, 	$Y_{k,1}$  and $Y_{\ell,2}$ are orthogonal with respect to the discrete inner product $\la \dot,\dot \ra_n$ by symmetry. For $k,\ell < n$, we have, due to the fact Gaussian quadrature is exact for all polynomials of degree $2n-1$, 
\begin{align*}
	\la Y_{k,1},Y_{\ell,1} \ra &= 2 \int_{-1}^1 p_k(w_{-{1 \over 2}}; x) p_\ell(w_{-{1 \over 2}}; x) {w(x) \over \sqrt{1-x^2}}  dx   \\
	&= 2 \sum_{j=1}^n w_j p_k(w_{-{1 \over 2}}; t_j) p_\ell(w_{-{1 \over 2}}; t_j)    = 	\la Y_{k,1},Y_{\ell,1} \ra_n.
\end{align*}
Similarly, for any $k , \ell \leq n-1$, we have
\begin{align*}
	\la Y_{k,2},Y_{\ell,2} \ra &=  \int_{-1}^1 p_{k-1}(w_{{1 \over 2}}; x) p_{\ell-1}(w_{{1 \over 2}}; x)  (1-x^2) {w(x) \over \sqrt{1-x^2}}  dx  \\
	&= \sum_{j=1}^n w_j (1-t_j^2) p_{k-1}(w_{{1 \over 2}}; t_j) p_{\ell-1}(w_{{1 \over 2}}; t_j)   = 	\la Y_{k,2},Y_{\ell,2} \ra_n.
\end{align*}
as $p_{k-1}(w_{{1 \over 2}}; x) p_{\ell-1}(w_{\f12}; x)  (1-x^2)$ is of degree $k+\ell \leq 2 n -2$. 

This leaves only the case $k = \ell = n$. Since $x_j$ are precisely the roots of $p_n(w_{1/2})$, we actually have that 	$\la Y_{n,1},Y_{n,1} \ra_n = 0$. On the other hand, $(1-x^2) p_{n-1}(w_{1/2}; x)^2$ cannot vanish on $x_j$, as if it did  $p_{n-1}(w_{1/2}; x)$ would also vanish at the $n$ points $x_j$ (as $x_j$ do not contain $\pm 1$), hence would be zero.

\end{proof}

Before proceeding, we note that suitable Gauss--Lobatto points can be used to interpolate by the basis $\{Y_{j,1}\}_{j=0}^{n-1}, \{Y_{j,2}\}_{j=1}^{n-1}$ (if one endpoint is fixed) or $\{Y_{j,1}\}_{j=0}^{n}, \{Y_{j,2}\}_{j=1}^{n-1}$ (if two endpoints are fixed). Indeed, when $w(x) = 1$ these variants give different versions of the real-valued discrete Fourier transform.

A similar construction readily translates to each of the other cases. We omit the proofs as they are essentially the same as the circle case.

\begin{cor}[Parabola] \label{prop:parabola-interp1}
Let $\{t_j\}_{j=1}^n$ and $\{w_j\}_{j=1}^n$ be the $n$-point Gaussian quadrature nodes and weights with respect to $w_{-\f12}$ defined on ${\mathbb R}^+$, as in Theorem~\ref{thm:parabolabasis}. Then the $2n$ polynomials $\{Y_{j,1}\}_{j=0}^{n-1}, \{Y_{j,2}\}_{j=1}^n$ are orthogonal with respect to the discrete inner product
$$
\la f,g\ra_n := \sum_{j=1}^n w_j \left[f(x_j, y_j)  g(x_j, y_j) +  f(-x_j, y_j)  g(-x_j, y_j)\right],
$$
where $x_j = \sqrt{t_j}$ and $y_j = t_j$. 
\end{cor}

\begin{cor}[Hyperbola, two branches]\label{cor:twocutinterp}
Let $\{t_j\}_{j=1}^n$ and $\{w_j\}_{j=1}^n$ be the $n$-point Gaussian quadrature nodes and weights with respect to $w_0$ defined on ${\mathbb R}^+$, as in Theorem~\ref{thm:TwoBranch}. Then the $2n$ polynomials $\{Y_{j,1}\}_{j=0}^{n-1}, \{Y_{j,2}\}_{j=1}^n$ are orthogonal with respect to the discrete inner product
$$
\la f,g\ra_n := \sum_{j=1}^n w_j \left[f(x_j, y_j)  g(x_j, y_j) +  f(-x_j, y_j)  g(-x_j, y_j)\right],
$$
where $x_j = \sqrt{t_j^2+1}$ and $y_j = t_j$.
\end{cor}

\begin{cor}[Hyperbola, one branch]\label{prop:hypbola-interp1}
Let $\{t_j\}_{j=1}^n$ and $\{w_j\}_{j=1}^n$ be the $n$-point Gaussian quadrature nodes and weights with respect to $w_0$ defined on ${\mathbb R}_+$, as in Theorem~\ref{thm:OneBranch}. Then the $2n$ polynomials $\{Y_{j,1}\}_{j=0}^{n-1}, \{Y_{j,2}\}_{j=1}^n$ are orthogonal with respect to the discrete inner product
$$
\la f,g\ra_n := \sum_{j=1}^n w_j \left[f(x_j, y_j)  g(x_j, y_j) +  f(x_j, -y_j)  g(x_j, -y_j)\right],
$$
where $x_j = t_j$ and $y_j = \sqrt{t_j^2+1}$.
\end{cor}

\begin{cor}[Intersecting lines]\label{cor:intersectinglinesquad}
Let $\{t_j\}_{j=1}^n$ and $\{w_j\}_{j=1}^n$ be the $n$-point Gaussian quadrature nodes and weights with respect to $w$ defined on ${\mathbb R}$, as in Theorem~\ref{thm:OP-lines}. Then the $2n$ polynomials $\{Y_{j,1}\}_{j=0}^{n-1}, \{Y_{j,2}\}_{j=1}^n$ are orthogonal with respect to the discrete inner product
$$
\la f,g\ra_n := \sum_{j=1}^n w_j \left[f(t_j, 0)  g(t_j, 0) +  f(0, t_j)  g(0, t_j)\right].
$$
\end{cor}

\begin{cor}[Parallel lines]\label{prop:parallel-interp1}
Let $\{t_j\}_{j=1}^n$ and $\{w_j\}_{j=1}^n$ be the $n$-point Gaussian quadrature nodes and weights with respect to $w$ defined on ${\mathbb R}$, as in Theorem~\ref{thm:ParellelLines}. Then the $2n$ polynomials $\{Y_{j,1}\}_{j=0}^{n-1}, \{Y_{j,2}\}_{j=1}^n$ are orthogonal with respect to the discrete inner product
$$
\la f,g\ra_n := \sum_{j=1}^n w_j \left[f(t_j, -1)  g(t_j, -1) +  f(t_j, 1)  g(t_j, 1)\right].
$$
\end{cor}

\section{Applications}
 %\sotodo{This section has been mostly rewritten}

We now discuss three applications of orthogonal polynomials on quadratic curves. The first section makes connections with known results for the Fourier extension problem. The second section shows that the approximation of univariate functions with certain singularities or near-singularities can be recast as interpolation of non-singular functions on quadratic curves, leading to robust and accurate approximation schemes. The last section applies this construction to solving differential equations with nearly singular variable coefficients, in particular, Schr\"odinger's equation. 

\subsection{The Fourier extension problem}

Our first example considers the Fourier extension problem, that is, approximating $f(\theta)$ by a Fourier series
$$
f(\theta) \approx c_0 + \sum_{k=1}^n (c_k \cos k \theta + s_k \sin k \theta)
$$
where we are only allowed to sample $f$ in a sub-interval of $[-\pi,\pi)$, without loss of generality, the interval $[-h,h]$. Usually, it is assumed that the samples of $f$ are at evenly spaced points, but we consider a slightly more general problem where we are allowed to sample wherever we wish provided they are restricted to $[-h,h]$. Extensive work on the Fourier extension problem exists including fast and robust algorithms, see the recent work of Matthysen and Huybrechs \cite{MH} and references therein. 

As noted by Huybrechs in \cite{H}, the Fourier extension problem is closely related to expansion in orthogonal polynomials on an arc if we let $x = \cos \theta$ and $y = \sin \theta$. In fact, Huybrechs constructed quadrature schemes for determining these coefficients numerically, which are very close to the proposed interpolation method above. Here we make the minor observation that via the construction in Proposition~\ref{prop:circle-interp1} we can guarantee this expansion interpolates. This presents an attractive scheme for automatic determination of minimal degree needed for approximation, at which point the methods proposed in \cite{MH,H} can be used to re-expand in classical Fourier series.

\subsection{Interpolation of functions with  singularities}

We now consider an application of orthogonal polynomials on hyperbolic curves: approximation of (nearly-)singular functions via interpolation.  We consider the following three cases:
\begin{align*}
f(t) & = f(t, |t|), \\
f(t) & = f(t, \sqrt{t^2 + \epsilon^2}), \\% \hbox{and} \\ 
f(t) & = f(t, 1/t), 
\end{align*}
where $f(x,y)$ is a function smooth in $x$ and $y$, defined on the quadratic curves $y = x^2$, $y^2 =  x^2 + \epsilon^2$ or $xy = 1$, respectively. Thus instead of attempting to approximate the (nearly-)singular function $f(x)$ on the real line, we will approximate the non-singular function $f(x,y)$ on the appropriate quadratic curve via interpolation by orthogonal polynomials in two variables, as in the preceding section. This leads to faster convergence than conventional polynomial interpolation.

\subsubsection{Example 1: square-root singularity}

Consider functions with singularities like $\sqrt{t^2+\epsilon^2}$, for example
$$
f(t) = \sin(10t + 20\sqrt{t^2 + \epsilon^2})
$$
on the interval $-1 \leq t \leq 1$. 
We project $f(t)$ to the standard hyperbola using 
$$
f(x,y) = f(\epsilon y) = \sin(10 \epsilon  y + 20 \epsilon  x)
$$
on the one-branch hyperbola $x^2 = y^2 + 1$, $x > 1$ and $y \in [-\epsilon^{-1},\epsilon^{-1}]$ via the change of variables $t = y$ and $x = \sqrt{t^2 + \epsilon^2}$.

By Theorem~\ref{thm:OneBranch}, we can represent orthogonal polynomials in terms of univariate orthogonal polynomials. For concreteness, we choose: $w_0(t) =1$ and hence $w_1(t) = t^2-1$, where $1 \leq t \leq L$ for $L=\sqrt{1+1/\epsilon^2}$.  Here $w_0(t)$ is the Legendre weight, hence we have
$$
Y_{n,1}(x,y) = p_n(w_0; x) = P_n\left(2 {x-1 \over L-1} -1\right)
$$
On the other hand, $w_1(t)$ is not a classical weight but we can construct its orthogonal polynomials numerically via the Stieltjes procedure \cite[Section 2.2.3]{Gautschi}; indeed, because the weight itself is polynomial, the procedure is exact provided a sufficient number of quadrature points are used to discretize the inner product. This step needs to be repeated for each $\epsilon$ as the interval length depends on $\epsilon$. 

 Following the previous section, we find the coefficients of the polynomial $f_n(x,y)$ that interpolates $f(x,y)$ at the  $2n$ points $(x_j,\pm y_j)$, where    $y_j = \sqrt{x_j^2 - 1}$ and $x_j$ are the Gauss--Legendre points on the interval $[1,L]$. It follows that
$$
f_n(t) = f_n\left( \sqrt{1 + {t^2 \over \epsilon^2}}, {t \over \epsilon} \right)
$$
interpolates $f(t)$ at the points $\pm \epsilon y_j$.

\begin{figure}
 \begin{center}
 \includegraphics[width=.9\textwidth]{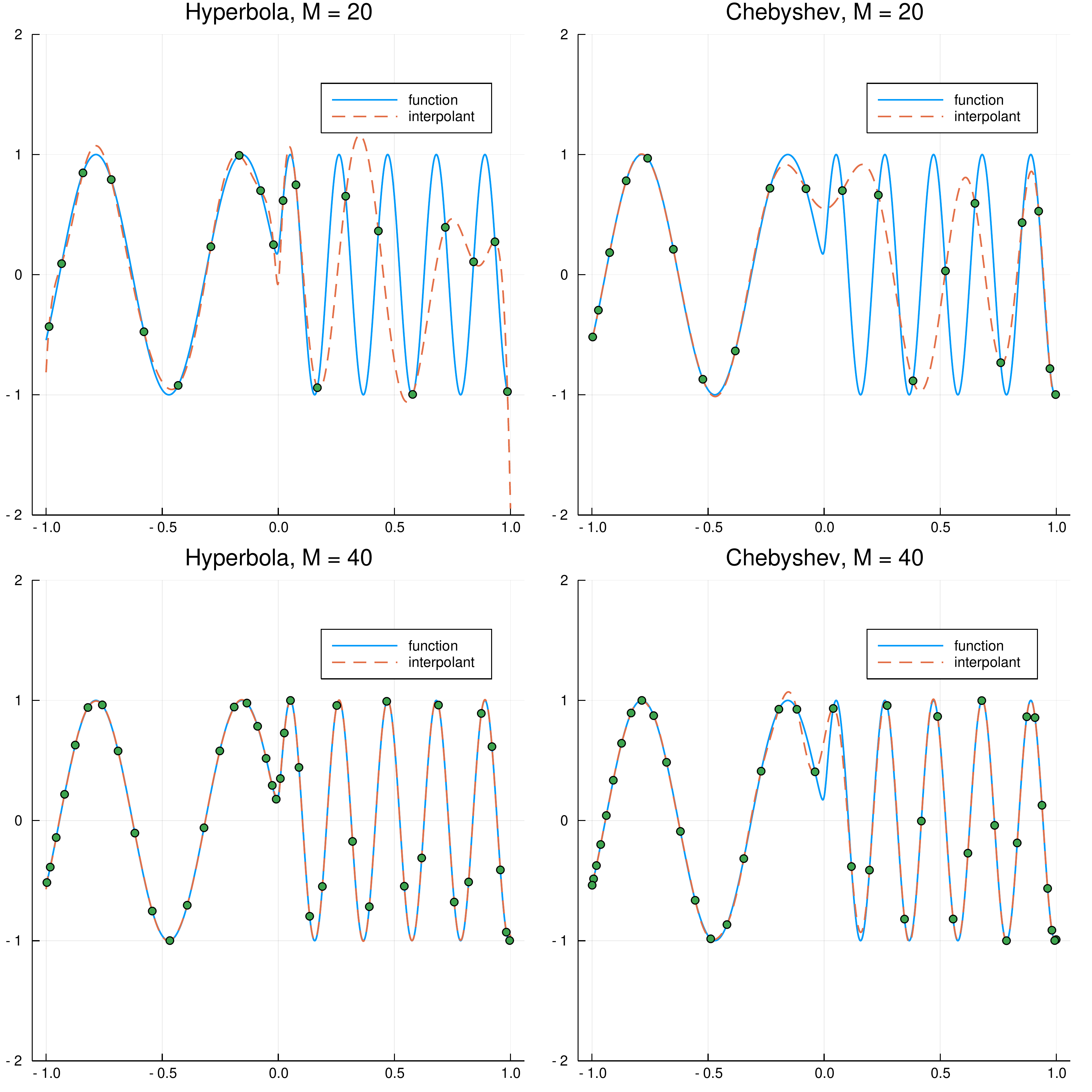}
  \caption{The $M = 2n = 20$ and $M= 2n = 40$ point interpolant of $\sin(10t + 20\sqrt{t^2 + \epsilon^2})$, for $\epsilon = 0.01$, comparing orthogonal polynomials on the hyperbola (left) to Chebyshev interpolation (right).    }\label{fig:hyperbola}
 \end{center}
\end{figure}

\begin{figure}
 \begin{center}
 \includegraphics[width=.9\textwidth]{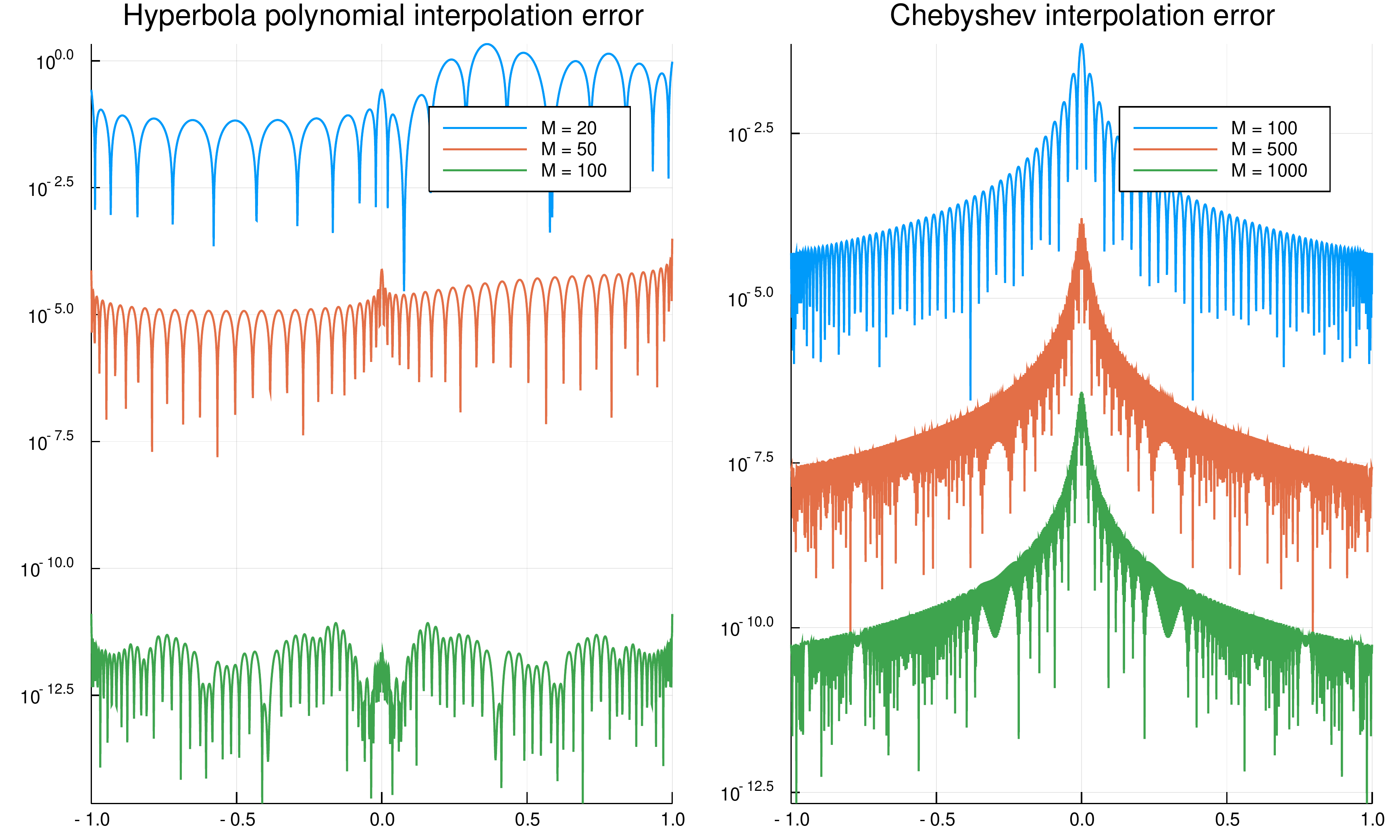}
  \caption{The convergence of the interpolant of $\sin(10t + 20\sqrt{t^2 + \epsilon^2})$, for $\epsilon = 0.01$ for increasing number of interpolation points $M$, comparing orthogonal polynomials on the hyperbola (left) to Chebyshev interpolation (right).    }\label{fig:hyperbolaerror}
 \end{center}
\end{figure}

In Figure~\ref{fig:hyperbola} we plot the interpolant for $20$ and $40$ interpolation points, for $\epsilon = 0.01$. In Figure~\ref{fig:hyperbolaerror} we plot the pointwise  error.  We include the Chebyshev interpolant for comparison: the approximation built from orthogonal polynomials on the hyperbola converges rapidly, despite the almost-singularity inside the domain, whereas the Chebyshev interpolant requires significantly more points to achieve the same accuracy.

\begin{figure}
 \begin{center}
 \includegraphics[width=.8\textwidth]{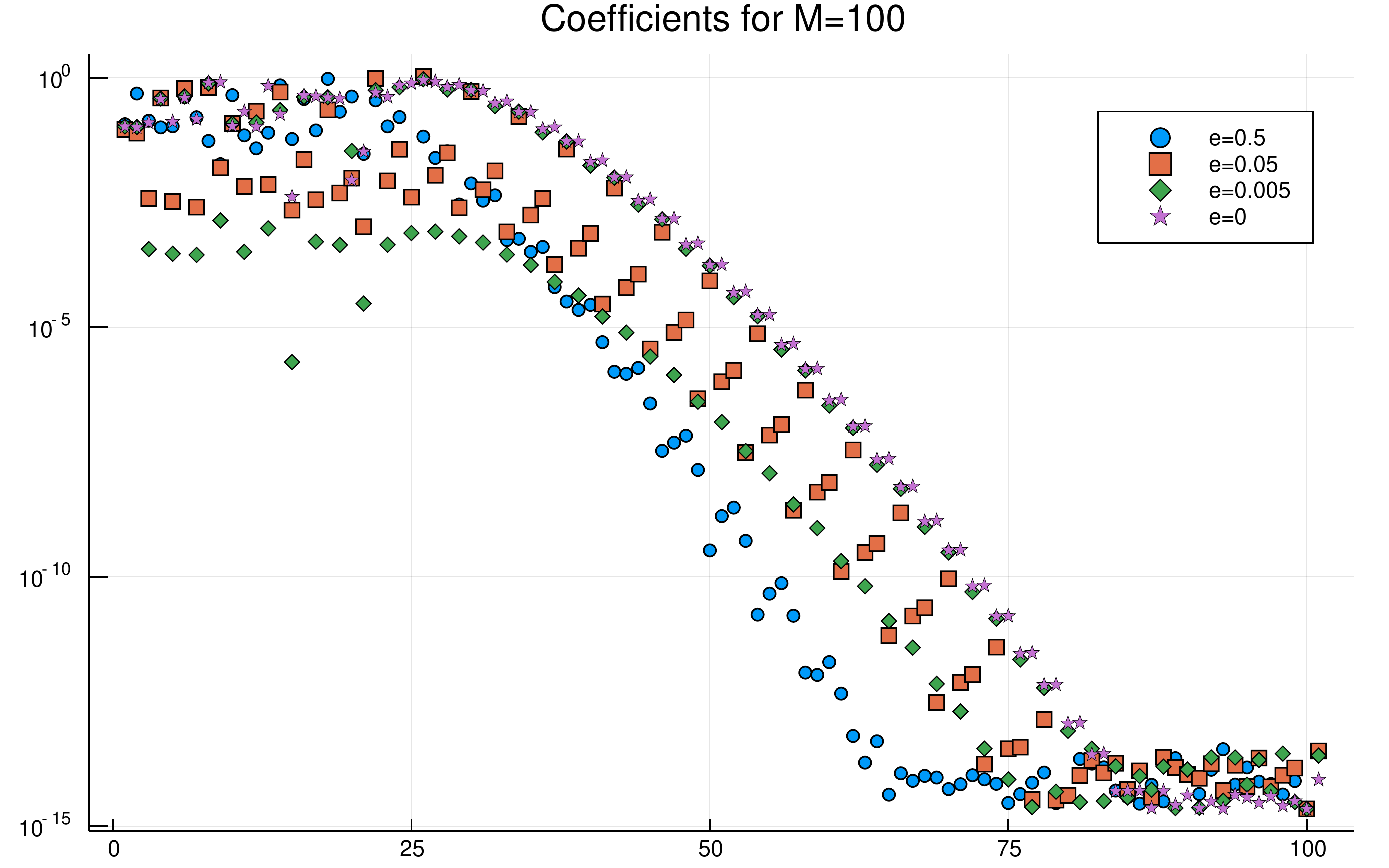}
  \caption{The calculated coefficients of the interpolating polynomial on the hyperbola for $100$ interpolation points and $\epsilon = 0.5, 0.05, 0.005, 0$.  The number of interpolation points is chosen sufficiently large that the interpolation coefficients approximate the true coefficients to roughly machine accuracy.   }\label{fig:hyperbolacfs}
 \end{center}
\end{figure}

In Figure~\ref{fig:hyperbolacfs} we plot the  coefficients of the interpolating polynomial for decreasing values of $\epsilon$.  We include $\epsilon = 0$: this may appear to be a degenerate limit as $L \rightarrow \infty$, but under an affine change of variables we see that the weights tend to the intersecting lines case with $w(t) = 1$ for $0 \leq t \leq 1$ in a continuous fashion, hence we can employ the construction in Corollary~\ref{cor:intersectinglinesquad}. We observe  super-exponential convergence for each value of $\epsilon$, with a rate uniformly bounded as $\epsilon \rightarrow 0$. This compares favourably with  polynomial interpolation at Chebyshev points,  which degenerates to slow algebraic convergence as $\epsilon \rightarrow 0$.

\begin{figure}
 \begin{center}
 \includegraphics[width=.7\textwidth]{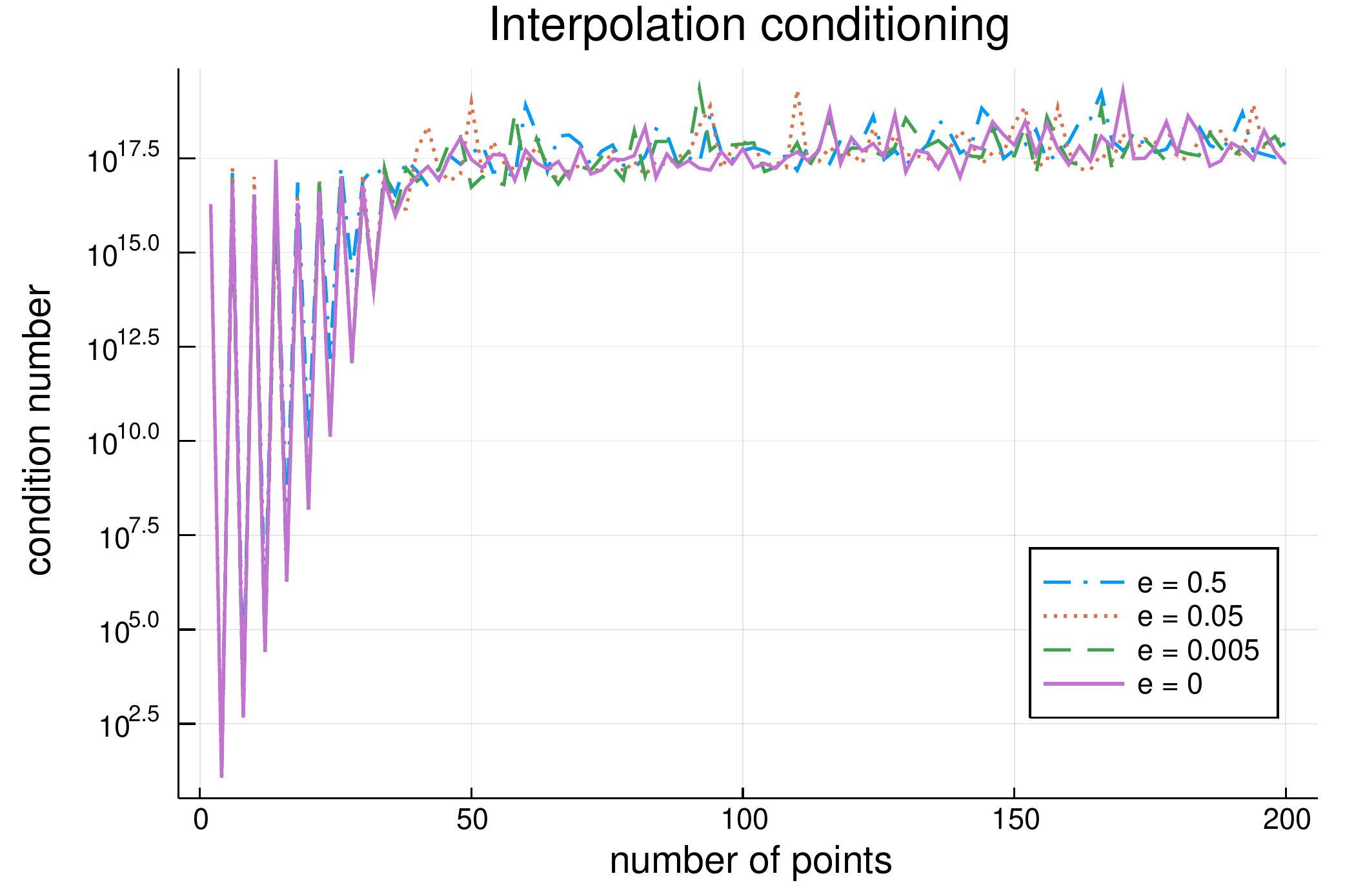}  \includegraphics[width=.7\textwidth]{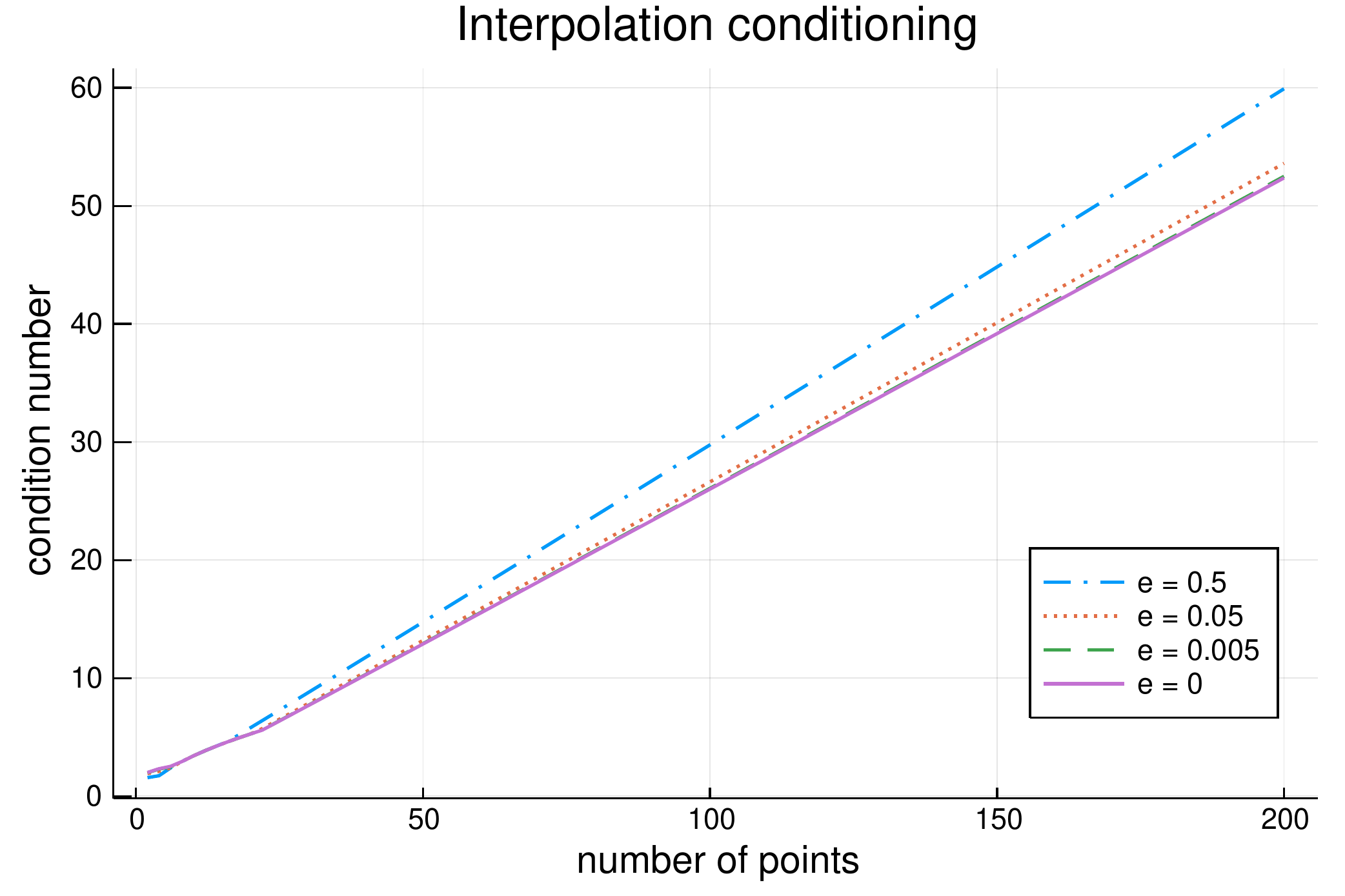}
  \caption{The condition number of the interpolation (Vandermonde) matrix associated to a na\"\i ve basis consisting of $\{ T_j(t) \}_{j=0}^{n-1}$ and $\{\sqrt{t^2 + \epsilon^2} T_j(t) \}_{j=0}^{n-1}$ at $2n$ Chebyshev points (left, log-scale) compared to the interpolation (Vandermonde) matrix using OPs on the hyperbola at the associated $2n$ interpolation points (right, standar-scale) for  $\epsilon = 0.5, 0.05, 0.005, 0$.  The proposed scheme has only linear growth in condition number, compared to exponential growth for the na\"\i ve basis.}\label{fig:cond}
 \end{center}
\end{figure}

%\sotodo{This paragraph addresses editors point 2}
Note that there is an alternative  approach of approximating the solution using standard bases, for example, we 
could na\"\i vely write
$$
f(t) \approx \sum_{k = 0}^{n-1} f_k^1 T_k(t)  + \sqrt{t^2 + \epsilon^2}  \sum_{k = 0}^{n-1} f_k^2 T_k(t).
$$
The benefit of the proposed construction in terms of orthogonal polynomials on the hyperbola is that the orthogonality ensures that the interpolation is well-conditioned, whereas the standard approach results in {\it exponentially} bad conditioning, see Figure~\ref{fig:cond}.  Recent advances in approximation with frames can overcome the issues with ill-conditioning for the approximation problem (see \cite{MH} and references), but at the expense of no longer interpolating the data.  It remains open how to use these techniques for the solution of differential equations as in Example 3.

\subsubsection{Example 2: essential singularity} 

We now consider a function that is smooth in $t$ and $1/t$, for example,
$$
f(t) = \sin(t + 2/t)
$$
on the real line. We project $f$ to  the two-branch hyperbola  $x^2 = y^2 + 1$ by the change of variables $t = x-y$:
$$
f(x,y) = f(x-y) = \sin\left(  x- y +2(x + y)\right)
$$
where we use the fact that if $t = x-y$ then $t^{-1} = x+y$ since $t t^{-1} = (x-y) (x+y) = x^2 -y^2 = 1$. 

We repeat the procedure used in the previous example, this time with the Hermite weight:
\begin{enumerate}
\item Construct $Y_{n,1}(x,y)$ using Hermite polynomials, orthogonal with respect to $w(t) = e^{-t^2}$.
\item Construct $Y_{n,2}(x,y)$ using polynomials orthogonal with respect to $w(t) = (1+t^2) e^{-t^2}$. This construction is performed numerically using the Stieltjes procedure, with Gauss--Hermite quadrature to discretize the inner product.  
\item Calculate the coefficients of the interpolation polynomial $f_M(x,y)$, which equals $f(x,y)$  at the $M = 2 \nu$ points $\{(\pm x_j,y_j)\}_{j=1}^\nu$, where $y_j$ are the $\nu$ Gauss--Hermite quadrature points and $x_j = \sqrt{y_j^2 + 1}$, via Corollary~\ref{cor:twocutinterp}.
\item The function
$$
f_M(t) = f\left({t + t^{-1} \over 2}, {t - t^{-1} \over 2}\right)
$$
interpolates $f(t)$ at the $M$ points $\{ \pm x_j - y_j \}_{j=1}^\nu$. 
\end{enumerate}

\begin{figure}
 \begin{center}
 \includegraphics[width=.7\textwidth]{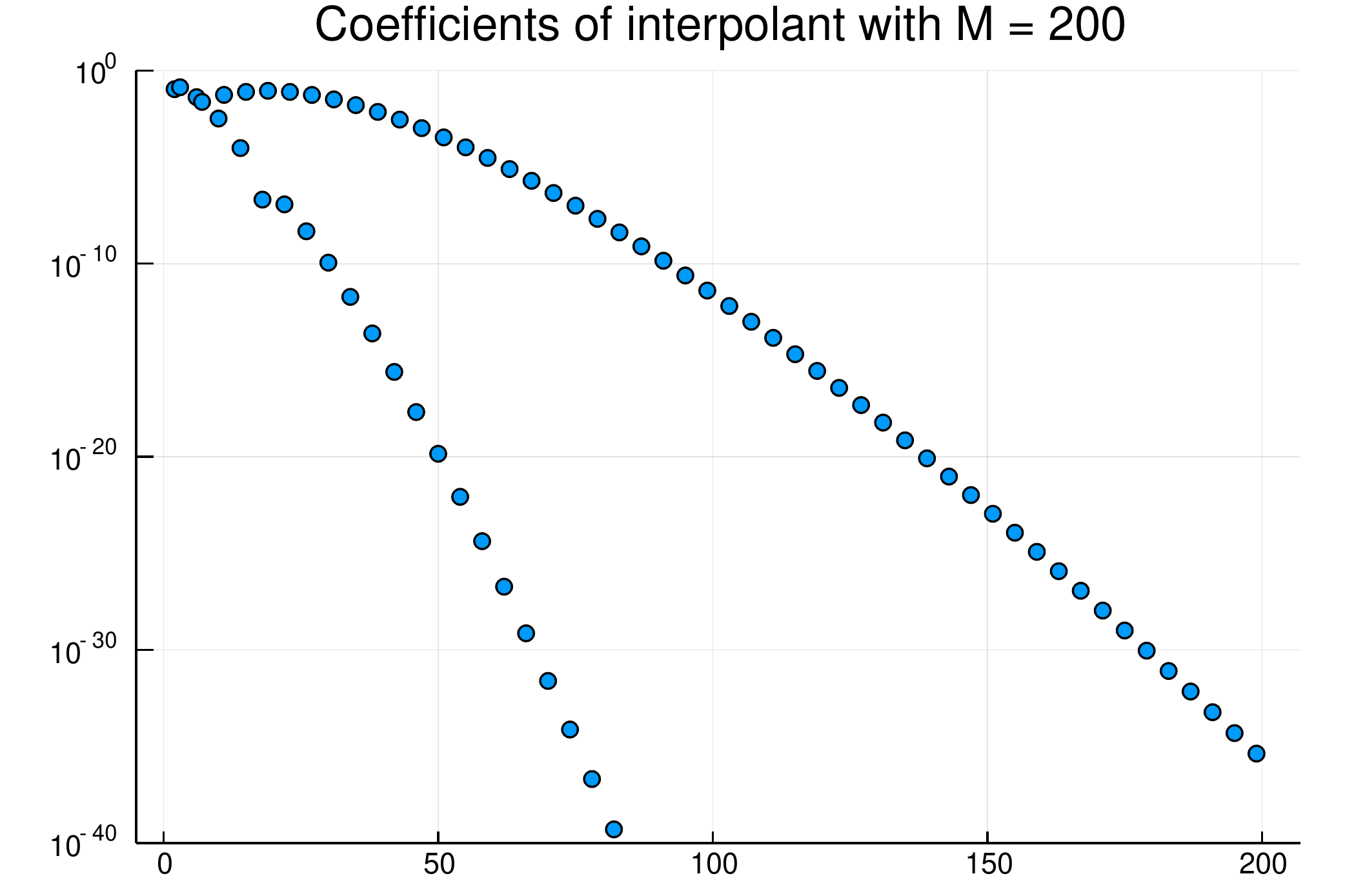}
  \includegraphics[width=.7\textwidth]{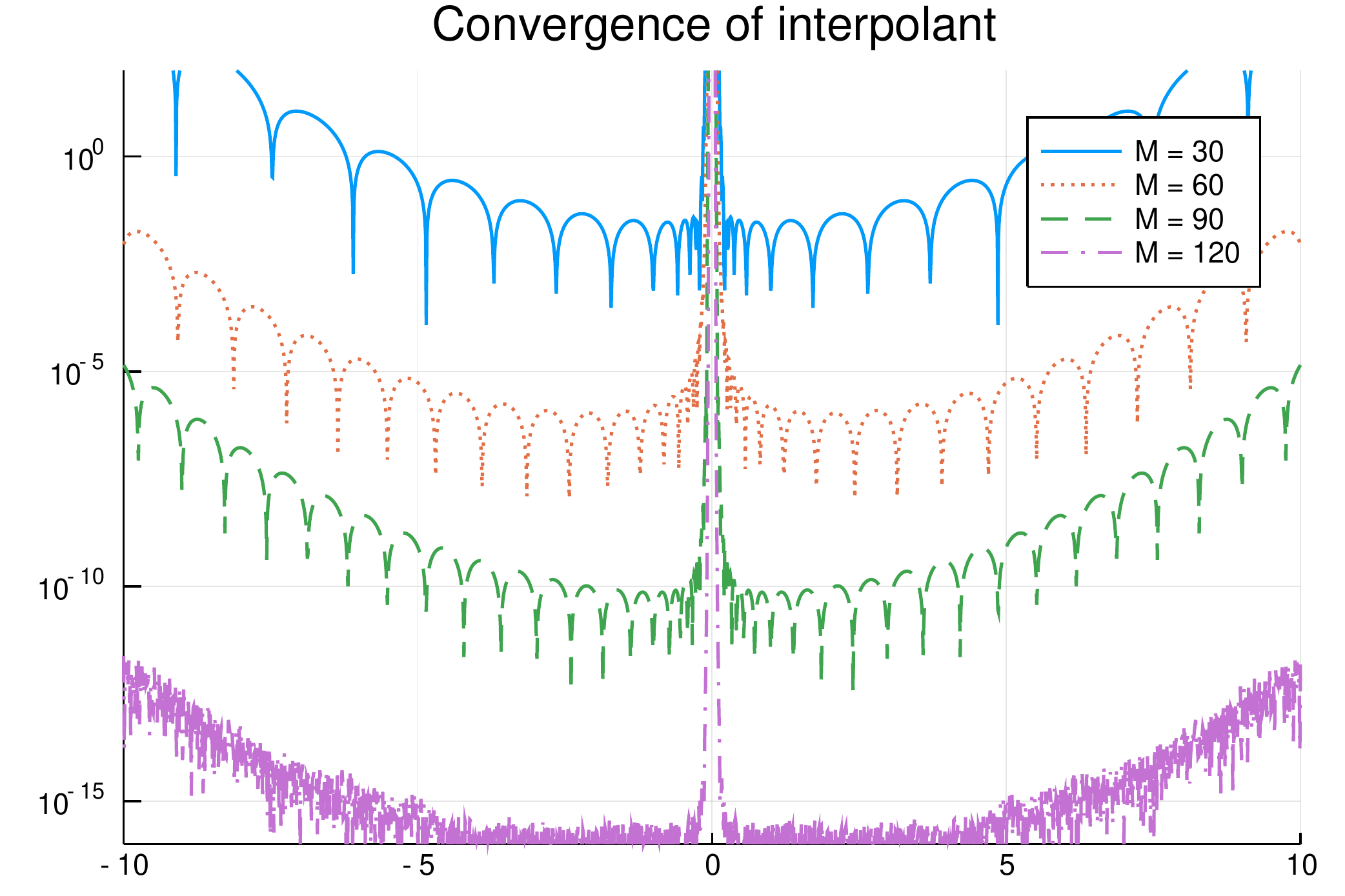}
  \caption{The coefficients of the 200 point interpolant of $\sin(x +2/x)$ (top), using high precision arithmetic. Convergence of the interpolants as $M$ increases (bottom).  }
  \label{fig:pole}
 \end{center}
\end{figure}

In Figure~\ref{fig:pole}, we plot the decay in  the coefficients  of the interpolating polynomial of   $\sin(t + 2/t) $ calculated via the proposed quadrature rule, showing an exponential decay rate. We also plot the pointwise convergence of the interpolant, showing exponentially fast uniform convergence in compact sets bounded  away from the singularity $t = 0$. Note the convergence is uniform everywhere (including at the singularity) in a weighted sense.

\subsection{Solving differential equations with singular variable coefficients}

As a final example, we consider the solution of Schr\"odinger's equation with a nearly singular well:
$$
-h^2 u''(t) + V(t) u =\lambda u
$$
where we take $h = 0.1$ and $V(t) = \sqrt{t^2 + \epsilon^2} + (t-0.1)^2$, and use Dirichlet conditions on the interval $-3 \leq t \leq 3$. As $\epsilon \rightarrow 0$ the potential becomes increasingly degenerate until it is non-differentiable and approaches $|t| + (t-0.1)^2$. For large $\epsilon$  and $\epsilon = 0$ standard spectral method techniques are applicable using Chebyshev or piecewise Chebyshev approximations as implemented in Chebfun (via {\tt quantumstates} command) \cite{ChebfunGuide} or ApproxFun.jl \cite{ApproxFun}. However, the computations quickly become infeasible as $\epsilon$ becomes small but non-zero: for example, with $\epsilon = 0.1$ we require roughly 1000 coefficients (hence the solution of a $1000 \times 1000$ eigenvalue problem) to calculate the smallest eigenvalue, while with $\epsilon = 0.01$ we were unable to succeed even using $5000$ coefficients.

As an alternative, we will use OPs on a hyperbola via the change of variables $t = y$ and $ x = \sqrt{t^2 + \epsilon^2}$ so that the potential is well-resolved by our basis. We then discretise the differential equation using a collocation system using $2n-2$ points as in the interpolation scheme above, where $2n$ is the total number of coefficients. For simplicity of implementation, the second derivatives of the basis are calculated using automatic-differentiation (via TaylorSeries.jl \cite{TaylorSeriesjl}), however, we could also achieve this analytically by differentiating the three-term recurrence: that is, if we have the recurrence coefficients of the univariate orthogonal polynomials we can differentiate the relationship
$$
(J-x) \begin{pmatrix} p_0(x) \\ p_1(x) \\ \vdots \end{pmatrix} = 0
$$
to find a recurrence relationship for the derivative,
$$
(J - x) \begin{pmatrix} p_0'(x) \\ p_1'(x) \\ \vdots \end{pmatrix} =  \begin{pmatrix} p_0(x) \\ p_1(x) \\ \vdots \end{pmatrix},
$$
and thereby determine $p_n'(x)$ by forward substitution.   Finally, we impose Dirichlet conditions as two extra rows, leading to a generalised finite-dimensional eigenvalue problem. For $\epsilon = 0.1$ we confirm the calculation matches that of a standard Chebyshev spectral method (to at least 10 digits).

\begin{figure}
 \begin{center}
 \includegraphics[width=.7\textwidth]{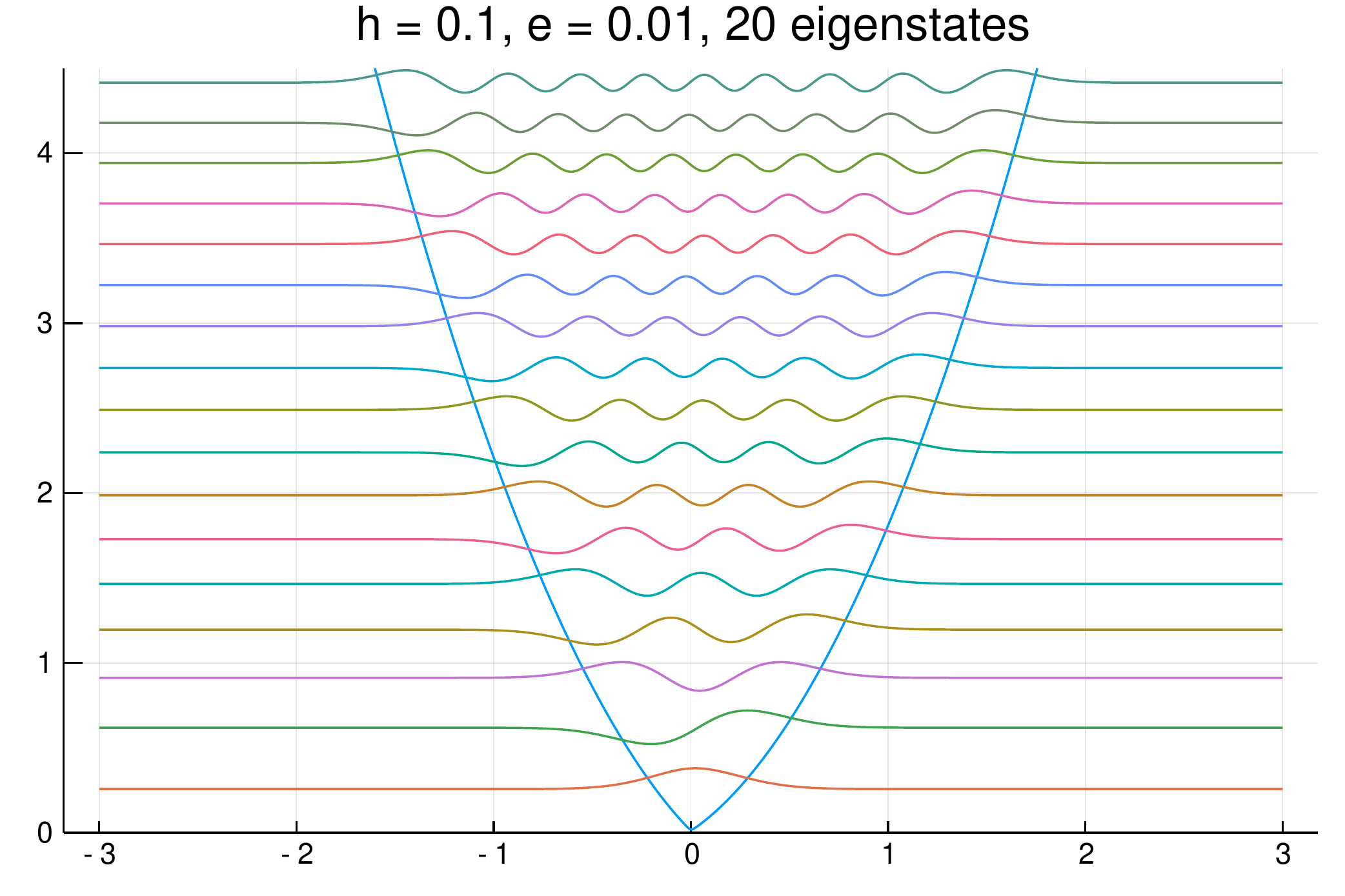}
  \caption{The first 20 eigenstates for Schr\"odinger equation with nearly singular potential $V(t) = \sqrt{t^2 + \epsilon^2} + (t-0.1)^2$  with $\epsilon = 0.01$.  The eigenstates   are shifted by the corresponding eigenvalue and scaled to fit inside the figure. }
  \label{fig:eigenstates}
 \end{center}
\end{figure}

\begin{figure}
 \begin{center}
 \includegraphics[width=.7\textwidth]{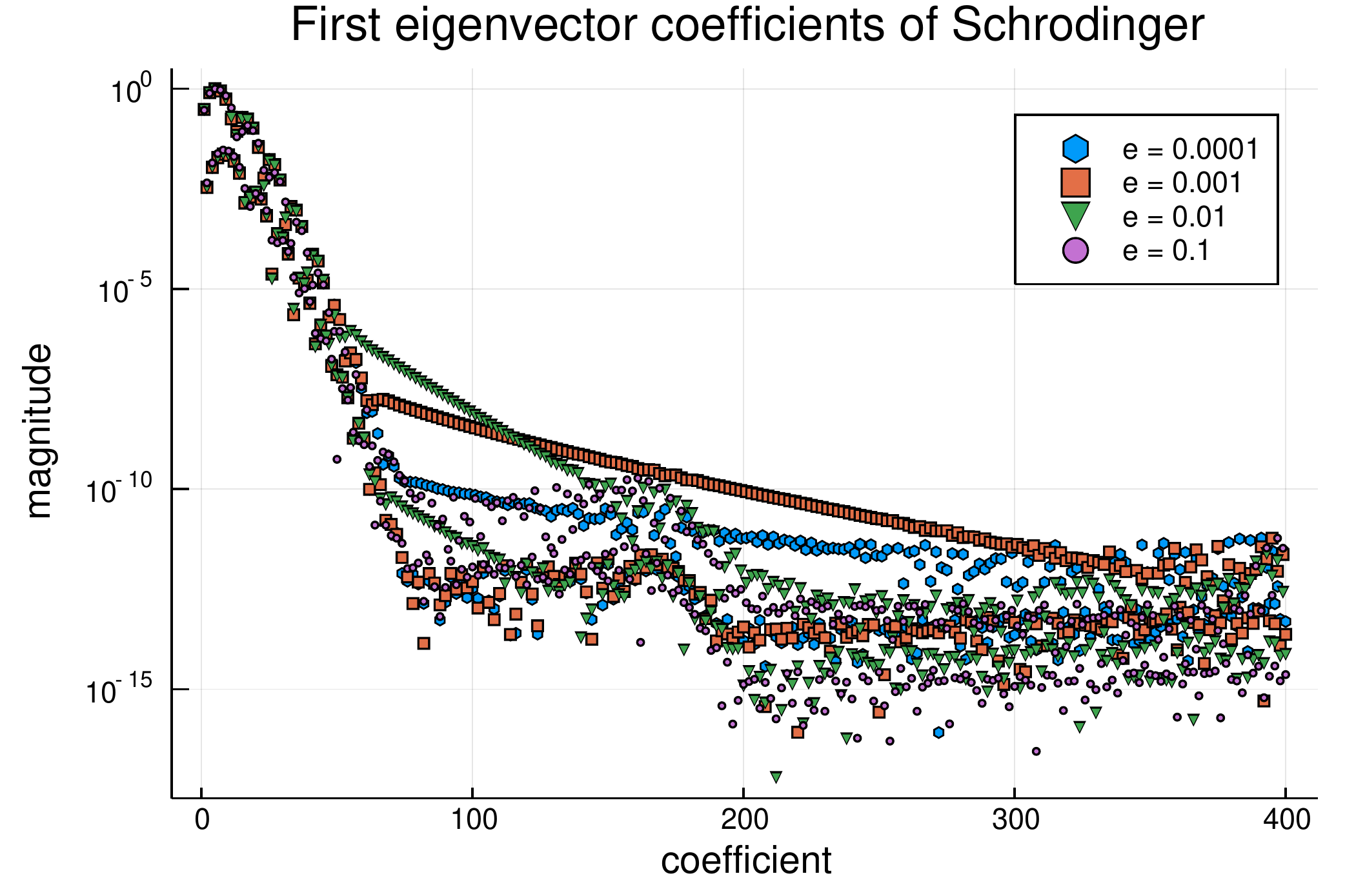}
  \caption{The coefficients of the smallest energy eigenstate in OPs on the hyperbola for Schr\"odinger equation with nearly singular potential $V(t) = \sqrt{t^2 + \epsilon^2} + (t-0.1)^2$ for $\epsilon = 0.1, 0.01, 0.001, 0.0001$, using a $400 \times 400$ collocation system.  }
  \label{fig:eigenfunctiondecay}
 \end{center}
\end{figure}

In Figure~\ref{fig:eigenstates} we plot the first 20 eigenstates with $\epsilon = 0.01$.  In Figure~\ref{fig:eigenfunctiondecay} we depict the calculated coefficients of the smallest energy eigenstate in OPs on the hyperbola. We can see that the first eigenstate continues to be resolved even for $\epsilon = 0.0001$, though there is some plateaux effect as $\epsilon \rightarrow 0$. While the decay rate of the plateaux degenerates, the maginitude of the plateaux improves  with smaller $\epsilon$ hence we continue to achieve high accuracy computations.

\section{Future work}

We have explored orthogonal polynomials on quadratic curves, and shown for weights satisfying certain symmetry properties that they can be constructed explicitly using  orthogonal polynomials in one variable. We have used these orthogonal polynomials as a basis to interpolate functions with singularities of the form $|x|$, $\sqrt{x^2 + \epsilon^2}$, and $1/x$, where the coefficients of the interpolant are determined via quadrature rules. Exponential or super-exponential convergence was observed in each case. We showed that the construction is also applicable to solving differential equations with nearly singular variable coefficients. 

Further applications of this work are to function approximation, quadrature rules, and solving differential equations involving other quadratic singularities,  for example, functions that are smooth in $x$ and $\sqrt x$ on $[0,1]$.  Such functions arise naturally in numerical methods for half-order Riemann--Liouville and Caputo fractional differential equations \cite{HO}. Orthogonal polynomials on a half-parabola---that is $y^2 = x$, $0 < x < 1$---would form a natural and convenient basis for approximating such functions. However, we are left with the task of constructing orthogonal polynomials for a weight that does not satisfy the requisite symmetry properties that allowed for reduction to univariate orthogonal polynomials.

Orthogonal polynomials on quadratic curves are also of use in partial differential equations on exotic geometries, for example, they have been recently used by Snowball and the first author for solving partial differential equations on disk slices and trapeziums \cite{SO}. %\sotodo{This sentence was added to address the editor comment that use in PDEs was far fetched}
 Finally, we mention that the results can be extended to higher dimensions, in particular to quadratic surfaces of revolution \cite{OX2}. It is also likely possible to reduce orthogonal polynomials on higher dimensional geometries satisfying suitable symmetry relationships to one-dimensional orthogonal polynomials, as is the case for multivariate orthogonal polynomials with weights that are invariant under symmetry reductions \cite{DX}. 

A natural question that arises is the structure of orthogonal polynomials on general (higher than quadratic) algebraic curves and surfaces. A closely related question is the structure of orthogonal polynomials inside algebraic curves and surfaces. It is unlikely that we will be able to reduce orthogonal polynomials on general algebraic geometries to one variable orthogonal polynomials, so other techniques will be necessary. Understanding this structure could lead to further applications in function approximation and numerical methods for partial differential equations, as well as theoretical results.

\bigskip\noindent
{\bf Acknowledgment.} We thank the referee for her/his carful reading and constructive comments.

\end{document}